\documentclass{amsart}
\usepackage[margin=1.3in]{geometry}
\usepackage{color}
\usepackage{float}
\usepackage{amsmath}
\usepackage{amssymb}
\usepackage{amsthm}
\usepackage{mathrsfs}
\usepackage{enumerate}
\usepackage{appendix}
\usepackage{tikz}
\usepackage{bm,amsbsy}

\numberwithin{equation}{section}

\newtheorem{thm}{Theorem}[section]
\newtheorem*{thm*}{Theorem}
\newtheorem{prop}[thm]{Proposition}
\newtheorem{lemma}[thm]{Lemma}
\newtheorem*{lemma*}{Lemma}

\newtheorem{cor}[thm]{Corollary}

\theoremstyle{definition}

\theoremstyle{remark}
\newtheorem{remark}[thm]{Remark}

\definecolor{pink}{rgb}{1,.2,.6}
\definecolor{orange}{rgb}{0.7,0.3,0}
\definecolor{blue}{rgb}{.2,.6,.75}
\definecolor{green}{rgb}{.4,.7,.4}
\definecolor{purple}{RGB}{127,0,255}

\newcommand{\bF}{\mathbb{F}}

\newcommand{\bQ}{\mathbb{Q}}
\newcommand{\bR}{\mathbb{R}}

\newcommand{\bZ}{\mathbb{Z}}

\newcommand{\modd}[1]{\; ( \text{mod} \; #1)}

\newcommand{\sing}{{\rm sing }}

\newcommand{\lam}{\lambda}

\newcommand{\Lbf}{\mathbf{L}}
\newcommand{\Hbf}{\mathbf{H}}
\newcommand{\Mbf}{\mathbf{M}}
\newcommand{\Nbf}{\mathbf{N}}

\newcommand{\tbf}{\mathbf{t}}

\newcommand{\Xbf}{\mathbf{X}}
\newcommand{\Ybf}{\mathbf{Y}}
\newcommand{\Zbf}{\mathbf{Z}}

\newcommand{\abf}{{\bf a}}

\newcommand{\hbf}{{\bf h}}

\newcommand{\ubf}{{\bf u}}
\newcommand{\vbf}{{\bf v}}
\newcommand{\wbf}{{\bf{w}}}

\newcommand{\xbf}{{\bf x}}
\newcommand{\ybf}{{\bf y}}
\newcommand{\zbf}{{\bf z}}

\newcommand{\beq}{\begin{equation}}
\newcommand{\eeq}{\end{equation}}

\newcommand{\ba}{\begin{align}}
\newcommand{\ea}{\end{align}}

\makeatletter
\def\@tocline#1#2#3#4#5#6#7{\relax
  \ifnum #1>\c@tocdepth 
  \else
    \par \addpenalty\@secpenalty\addvspace{#2}%
    \begingroup \hyphenpenalty\@M
    \@ifempty{#4}{%
      \@tempdima\csname r@tocindent\number#1\endcsname\relax
    }{%
      \@tempdima#4\relax
    }%
    \parindent\z@ \leftskip#3\relax \advance\leftskip\@tempdima\relax
    \rightskip\@pnumwidth plus4em \parfillskip-\@pnumwidth
    #5\leavevmode\hskip-\@tempdima
      \ifcase #1
       \or\or \hskip 1em \or \hskip 2em \else \hskip 3em \fi%
      #6\nobreak\relax
    \hfill\hbox to\@pnumwidth{\@tocpagenum{#7}}\par
    \nobreak
    \endgroup
  \fi}
\makeatother

\newcommand{\genlegendre}[4]{%
  \genfrac{(}{)}{}{#1}{#3}{#4}%
  \if\relax\detokenize{#2}\relax\else_{\!#2}\fi
}

\makeatletter
\let\@@pmod\pmod
\DeclareRobustCommand{\pmod}{\@ifstar\@pmods\@@pmod}
\def\@pmods#1{\mkern4mu({\operator@font mod}\mkern 6mu#1)}
\makeatother

\usepackage{enumitem}

\begin{document}

\title{Short character sums of inhomogeneous polynomials}

\author[Rena Chu]{Rena Chu}

\address{Georg-August-Universität Göttingen, Bunsenstraße 3-5, 37073 Göttingen, Germany}
\email{rena.chu@mathematik.uni-goettingen.de}

\begin{abstract}
   Let $p$ be a prime. We prove nontrivial bounds on short sums of Dirichlet characters mod $p$ evaluated at a class of polynomials, not necessarily homogeneous, in $n$ variables and of degree $k$. For large $n$, we further achieve nontrivial bounds for sums over boxes with side-lengths as short as $p^{1/(k-1)+\varepsilon}$, which breaks past the Burgess barrier of $p^{1/4+\varepsilon}$ as soon as $k\geq 6$. In the proof, we develop a new variation of the Burgess amplification method that reduces the problem to bounding additive character sums. This is the first case in which a Burgess-type method succeeds in the inhomogeneous setting.
\end{abstract}

\maketitle

\section{Introduction}
Let $p$ be a prime and $\chi$ be a nonprincipal Dirichlet character mod $p$. One central question in analytic number theory is to bound from above the character sum
    \begin{align*}
        S(N,H) = \sum_{x\in (N,N+H]}\chi(x).
    \end{align*}
The Pólya-Vinogradov bound gives $S(N,H)\ll p^{1/2}\log p$, which is nontrivial for lengths $H$ greater than $p^{1/2}\log p$. In the complementary regime of short character sums, where $H\ll p^{1/2}$, the celebrated work of Burgess \cite{Bur57,Bur62b,Bur62a,Bur63} provides a family of bounds $S(N,H)\ll H^{1-1/r}p^{(r+1)/4r^2}\log p$, for $r\geq 1$, that is nontrivial for $H\gg p^{1/4+\kappa}$ for any $\kappa>0$.

    One generalization of this problem is to bound sums of characters with polynomial arguments.
    Let $n,k\geq 1$ and $F(\Xbf)\in \bZ[X_1,...,X_n]$ be a polynomial of degree $k$. For any $\Nbf=(N_1,...,N_n)$ and $\Hbf=(H_1,...,H_n)$, consider the character sum
    \begin{align*}
        S(F;\Nbf,\Hbf) = \sum_{\substack{x_1,...,x_n \\ x_i\in (N_i,N_i+H_i]}} \chi(F(\xbf)).
    \end{align*}
There is a long history of work on bounding such sums for various classes of polynomials $F$ as well as on the closely related problem of bounding multiplicative character sums over finite fields; these include \cite{DL63,Bur68,Kar68, Kar70,Gil73, Cha09,Cha09_binary, BC10, Kon10, HB16}. In 2020, Pierce and Xu \cite{PX20} proved the first nontrivial bounds for a generic class of forms in arbitrarily many variables and of arbitrary degree, for $H_i\gg p^{\beta_n+\kappa}$, where $\beta_n = 1/2 - 1/2(n+1)$, for all $1\leq i \leq n$. 
More recently, the author \cite{Chu25} built on the work of \cite{Gil73,BC10,Kon10} to prove nontrivial bounds for a class of forms that split over $\overline{\bF}_p$, for  $H_i\gg p^{1/4+\kappa}$. Up until our present work, these results constitute the best known bounds for character sums with polynomial arguments, where in particular, nontrivial bounds have thus far applied only to classes of homogeneous polynomials and to boxes with side-lengths as short as (but not shorter than) the original Burgess threshold, which has remained a longstanding barrier for bounds on $S(N,H)$, $S(F;\Nbf,\Hbf)$, and other generalizations such as sums over finite fields and mixed character sums.

This paper establishes the first nontrivial bounds for a large class of polynomials that are not necessarily homogeneous as well as nontrivial bounds for lengths shorter than the Burgess threshold of $p^{1/4+\kappa}$. More specifically, for a polynomial $F$ in $n$ variables and of degree $k$ such that $n$ is ``large'' and the singular locus of the leading form of $F$ is ``small'', we prove nontrivial bounds on $S(F;\Nbf,\Hbf)$ for $H_i$ as short as $p^{1/(k-1)+\kappa}$, for any $\kappa>0$. Here is our main result.

\begin{thm}\label{thm:main}
    Fix a prime $p$ and let $n\geq 1,k\geq 2$ be integers. Let $F\in \bF_p[X_1,...,X_n]$ be a polynomial in $n$ variables and of degree $k$, and let $F_k$ denote its leading form. 
 Suppose  
    \begin{align}\label{eqn:main_condition}
         n-\dim_{\overline{\bF}_p}\sing F_k \geq 2^{k-2}(k-1)^2(k+1),
    \end{align}
    where $\sing F_k$ denotes the singular locus in affine space of the hypersurface defined by $F_k=0$ over $\bF_p$. Let $\Nbf=(N_1,...,N_n)$ and $\Hbf=(H_1,...,H_n)$ with $H_i\leq p$, and define $\|\Hbf\|=H_1\cdots H_n$ and $H_\mathrm{min}=\min_i H_i$.  Let $\kappa>0$ and $H_{\mathrm{min}}\gg p^{1/(k-1)+\kappa}$.
    Then for any $\varepsilon>0$,
      \begin{align}\label{eqn:main_bound}
            S(F;\Nbf,\Hbf) \ll_{n,k,\varepsilon} \|\Hbf\| p^{-\kappa/2 + \varepsilon}.
        \end{align}
\end{thm}

In particular, we achieve nontrivial bounds for sums as short as $H_{\mathrm{min}}\gg p^{1/(k-1)+\kappa}$ for any $\kappa>0$, which falls in the regime of short sums as soon as $k\geq 4$ and beats the Burgess threshold of $p^{1/4+\kappa}$ as soon as $k\geq 6$, subject to the condition \eqref{eqn:main_condition}.  This is the first time that a nontrivial bound on a Dirichlet character sum with homogeneous polynomial arguments breaks past the Burgess barrier. It is also the first time that a Burgess-type method succeeds in an inhomogeneous setting. We further note that for fixed $n,k$ with $n\geq 2^{k-2}(k-1)^2(k+1)$, given a form $P\in \bZ[X_1,...,X_n]$ of degree $k$ that is nonsingular (in projective space) over $\overline{\bQ}$, its reduction modulo $p$ is nonsingular (in projective space) over $\overline{\bF}_p$ for all but finitely many primes $p$. Such polynomials are generic in the moduli space of forms in $n$ variables and of degree $k$.

We assume throughout that $p\gg_{n,k} 1$, since otherwise, \eqref{eqn:main_bound} holds trivially. In the proof, we present a more general version, where, for any $c>0$, we may replace the right-hand side of \eqref{eqn:main_condition} by $2^{k-2}(k-1)(k+1)c$ and the condition $H_{\min}\gg p^{1/(k-1)+\kappa}$ by $H_{\min}\gg p^{(1/k)(1+1/c)+\kappa}$. This reflects a trade-off between the number of variables required and the range of $H_i$ for which we obtain a nontrivial bound on $S(F;\Nbf,\Hbf)$. For example, the threshold $p^{(1/k)(1+1/c)}$ is less than $p^{1/2}$ as soon as $c>2/(k-2)$, but to dip below $p^{1/4}$ would require more variables, with $c>4/(k-4)$, and to reach $p^{1/(k-1)}$ as in Theorem \ref{thm:main} demands $c=k-1$.

By restricting $F$ to a class of polynomials with diagonal structure, we can weaken the condition on the number of variables relative to the degree.
  \begin{thm}\label{thm:main_diagonal}
     Fix a prime $p$ and let  $n\geq1, k\geq 5$ be integers. Let $F=G_1(X_1)+\cdots + G_{n}(X_{n})$, where each $G_i\in \bF_p[X_i]$ is of degree $k$.
     Let $\Nbf=(N_1,...,N_n)$ and $\Hbf=(H_1,...,H_n)$ with $H_i\leq p$. Suppose $n\geq (k+1)k(k-1)(k-2)/4$. Let $\kappa>0$ and $H_{\mathrm{min}}\gg p^{2/(k-2)  + \kappa}$. Then for any $\varepsilon>0$,
        \begin{align*}
            S(F;\Nbf,\Hbf) \ll_{n,k,\varepsilon} \|\Hbf\| p^{-\kappa/2+\varepsilon}.
        \end{align*}
 \end{thm}

This is shorter than the Burgess threshold as soon as $k\geq 11$. Here, the condition on the number of variables depends polynomially rather than exponentially in Theorem \ref{thm:main} on the degree. Assuming $p$ is large relative to $k$, one example of a polynomial whose leading form is nonsingular over $\overline{\bF}_p$ that satisfies the hypotheses of Theorem \ref{thm:main} but not of Theorem \ref{thm:main_diagonal} is $F=X_1X_2^{k-1} + X_2X_3^{k-1}+\cdots + X_{n-1}X_n^{k-1} + X_nX_1^{k-1} + G$, where $k\geq 3$, $n\geq 2^{k-2}(k-1)^2(k+1)$, and $G\in \bF_p[X_1,...,X_n]$ is any polynomial of degree at most $k-1$. Furthermore, an example of a polynomial whose leading form is singular and satisfies the hypotheses of Theorem \ref{thm:main} is $F=X_1X_2^{k-1} + X_2X_3^{k-1}+\cdots + X_{n-1}X_n^{k-1}+G$, where again $k\geq 3$, $n-1\geq 2^{k-2}(k-1)^2(k+1)$, and $G$ is any polynomial of degree at most $k-1$; here, the dimension of the singular locus of $F_k$ is one.

Our results do not apply to the class of forms considered in \cite{Chu25}, since we require $n$ to be much larger than $k$ while \cite{Chu25} assumes $n=k$. On the other hand, \cite{PX20} treats forms with no conditions on $n,k$; for every polynomial to which both Theorem \ref{thm:main} and \cite[Theorem 1.1]{PX20} apply,  we give nontrivial bounds for a wider range of lengths $H_i$, beating the threshold of $\beta_n=n/2(n+1)$ in \cite{PX20}, for all $n\geq 2$, as soon as $k\geq 5$ (and analogously $k\geq 9$ for Theorem \ref{thm:main_diagonal}).
We further obtain savings of $\kappa/2$, improving that in the shape of $\kappa^2$ (for small $\kappa$) from the previous works of \cite{PX20,Chu25}.
Finally, we remark that it still remains an open question to prove nontrivial bounds for short sums of (inhomogeneous) polynomials in one variable, e.g. $F(X)=X^2+1$.
%

\subsection{Method of proof}
We follow the general framework of amplification first developed by Burgess in his series of papers \cite{Bur57,Bur62b,Bur62a}. This has since been refined by \cite{GM10,HB13} and subsequently developed in multi-dimensional settings. The modern variations of this amplification that have appeared in previous known works on bounding $S(F;\Nbf,\Hbf)$ can be sorted into two distinct adaptations, which we call method I and method II (following the terminology in \cite{Chu25}). In both versions, the underlying idea is to redistribute copies of the original box of summation sufficiently many times so that a different variable (nearly) covers the entire set of residue classes mod $p$, thus producing a complete sum. The standard application of H\"older's inequality then separates a count $S_1$ of redundancies from redistribution and a complete sum $S_2$ of the $2r$-th moment of a character sum. Methods I and II differ in their redistribution techniques, which yield different shapes of $S_1$ and $S_2$, which in turn require different estimation strategies. In this paper, we develop a new variation that is distinct from both existing methods. Before presenting this, we give a brief overview of the two current methods (and refer the reader to \cite[\S 1.2]{Chu25} for more details).

In method I, which follows closely to the original technique of Burgess and then extended to the multi-dimensional case in \cite{Pie16,PX20}, one begins by fixing a prime $q$ different to the modulus $p$ and rewriting each element $\xbf \in (\Nbf,\Nbf+\Hbf]$ as a linear combination of $p$ and $q$, for example $\xbf = \abf p + \tbf q$, with $\abf \modd q^n$ and $t_i \in ((N_i-a_ip)/q, (N_i-a_ip+H_i)/q]$. Then by periodicity and multiplicativity of $\chi$ and homogeneity of $F$, $\chi(F(\abf p + \tbf q)) = \chi(q^{\deg F})\chi(F(\tbf))$. Averaging over $q$'s as well as the starting points and lengths of the range of $t_i$'s, one obtains, roughly,
\begin{align}\label{eqn:overview_tasks}
        S_1=\sum_{\substack{z_1,...,z_n \\ z_i \modd p}}\mathcal{A}(z_1,...,z_n)^2, 
        \qquad S_2=\sum_{\substack{z_1,...,z_n \\ z_i \modd p}} |\sum_{\substack{t_1,...,t_n\\ t_i\in (0,T_i]}} \chi(F(\tbf + \zbf))|^{2r},
    \end{align}
    where $\mathcal{A}(z_1,...,z_n)$ counts the number of primes $q$ (lying in a dyadic range) and $\abf\modd q ^n$ such that $z_i$ lies in $((N_i-a_ip-H_i)/q, (N_i-a_ip)/q]$ for $1\leq i \leq n$. (We refer the reader to \cite[\S 5]{PX20} for a precise deduction.)

Method II originates in the work of Karacuba \cite{Kar68,Kar70} (at least in the context of Dirichlet character sums) who bounded character sums over finite fields. The amplification here starts by shifting the variable $\xbf$ by a product $t\ybf$. (We mention also the work of Friedlander and Iwaniec \cite{FI93} who, following Iwaniec as recorded in \cite{Fri87}, developed a variant of method II to give an alternate proof of the classical Burgess bounds; in \S \ref{sec:amplification}, we apply a specific technique from this variant.) For characters with polynomial arguments,
method II requires more stringent conditions on $F$, more precisely, that the polynomial is a binary quadratic form in \cite{Cha09,Cha09_binary,HB16} and, in the $n$-dimensional case, that the polynomial has degree $n$ and splits over $\bF_p$ in \cite{BC10} and over $\overline{\bF}_p$ in \cite{Chu25} into linearly independent linear forms. For illustration, let $F$ be a single linear form $L$ so that $L(\xbf + t\ybf) = L(\xbf) +tL(\ybf)$, and then by multiplicativity, $\chi(L(\xbf + t\ybf)) = \chi(L(\ybf))\chi(t + L(\xbf)\overline{L}(\ybf))$. Average over many $t$'s and $\ybf$'s, and let $z=L(\xbf)\overline{L}(\ybf)$, so that summing over the variable $t$ generates a character sum that is linear in $t$ and independent of $\xbf,\ybf$. In the more general setting considered in \cite{Chu25}, one obtains
    \begin{align}\label{eqn:overview_tasks2}
           S_1=\sum_{\substack{z_1,...,z_s \\ z_i \in \bF_{p^{n_i}}}}\eta_F(z_1,...,z_s)^2, 
        \qquad S_2=\sum_{\substack{z_1,...,z_s \\ z_i \in \bF_{p^{n_i}}}} |\sum_{t\in (0,T]} \prod_{i=1}^s \psi_i(t+z_i)|^{2r},
    \end{align}
    where $\eta_F(z_1,...,z_s)$ counts the number of $(\xbf,\ybf)$'s in a box such that the simultaneous equations $z_i=\lambda_i(\xbf)\lambda_i(\ybf)^{-1}\in \bF_{p^{n_i}}$ for $1\leq i \leq s$ hold. (We omit the precise definitions of $s, n_i, \psi_i, \lambda_i$ and simply remark that these depend on $F$ and the $\psi_i$'s are multiplicative characters of $\bF_{p^{n_i}}$, where $n_1+\cdots + n_s = n$.)

    In the one-dimensional setting, the two methods described above produce the same bounds (see e.g. \cite[\S 8]{Pie20_super} for a comprehensive illustration of the modern formulation of method I and \cite[\S 12.4]{IK04} for method II), while in higher dimensions, there are significant differences. First is the class of polynomials to which each method is applicable:\ method I accommodates a generic class of forms for any $n,k$, while method II requires the form to satisfy $n=k$ and split into linearly independent factors. Second is the range of nontriviality:\ method I gives nontrivial bounds for $H\gg p^{\beta_n+\kappa}$, where recall $\beta_n=n/2(n+1)$, while method II recovers the Burgess threshold of $H\gg p^{1/4+\kappa}$. The key reason for this difference is that the character sum over $t$ in $S_2$ is $n$-dimensional in method I and one-dimensional in method II, as seen above in \eqref{eqn:overview_tasks} and \eqref{eqn:overview_tasks2}, respectively; from this perspective, the latter method is more efficient. Third is the principal difficulty of each approach:\ the heart of method I lies in estimating $S_2$ while the challenge of method II is bounding $S_1$. This distinction is apparent from the shapes of these respective sums: they each carry the dependence on $F$. Pierce and Xu \cite{PX20} carried out method I and applied a stratification result of \cite{Xu20} to bound $S_2$ which consists of complete, multi-dimensional, multiplicative character sums evaluated at highly singular polynomials, a case in which the Weil bound is not applicable. Their result produced bounds for the threshold of $p^{\beta_n}$, optimal for method I, and for the largest class of polynomials that is within reach of the method (see e.g. \cite[\S 8.2]{PX20}). 
Analogously, the author in \cite{Chu25} applied method II and proved an essentially optimal bound on $S_1$ in \eqref{eqn:overview_tasks2}, which is a multiplicative energy, using properties of lattices as introduced in \cite{Kon10} and a recursion technique from \cite{BC10}. By capitalizing on a correspondence between characters mod $p$ and characters of a finite field $\mathbb{F}_{p^n}$ (see e.g. \cite[\S 2.1]{Chu25}), she produced nontrivial bounds for the largest class of polynomials to which method II is applicable. 
    These works therefore have maximized the capabilities of the two methods, respectively, and any further substantial results would necessitate new ideas.

In this paper, we develop a novel version of the amplification method that extends the philosophy in method II of shifting by a product, to apply to any polynomial, a priori, without restrictions. This combines the efficiency of the second method with the flexibility of the first, and now we describe this new approach. We begin,  as in method II, with the shift $t\ybf$. The usual method II then exploits the assumed linear structure of $F$ to produce a character sum that is a product of linear factors in $t$; this reliance on linearity creates a fundamental barrier for method II to apply to other polynomials. 
The key to our new approach arises from the crucial observation that while such linearity is not present in a generic polynomial, we can nevertheless produce a polynomial in $t$. Indeed, by Taylor expansion (or the binomial theorem), we may rewrite 
    \begin{align}\label{eqn:method_taylor}
        F(\xbf + t\ybf) = \sum_{i=0}^k t^i  f_i(\xbf,\ybf), \qquad \text{where } f_i(\xbf,\ybf) = \sum_{\beta:|\beta| = i} \frac{(\partial^\beta F)(\xbf)}{\beta!} \ybf^\beta
    \end{align}
   and $\beta$ is a multi-index,
to obtain a degree $k$ polynomial in $t$. This step allows the amplification process to continue, with no condition so far on the shape of $F$. After the standard application of H\"{o}lder's inequality, we get 
    \begin{align}\label{eqn:S1S2_current}
    S_1=
    \sum_{\substack{z_0,...,z_k\\ z_i \modd p}}\eta(z_0,...,z_k)^2, \qquad  
     S_2=    \sum_{\substack{z_0,...,z_k\\ z_i \modd p}}|\sum_{t\in (0,T]}\chi(z_0 + z_1t + \cdots + z_k t^k)|^{2r},
\end{align}
where $\eta(z_0,...,z_k)$ counts the number of solutions $(\xbf,\ybf)$ in certain boxes to the simultaneous equations $z_i\equiv f_i(\xbf,\ybf) \modd p$ for $0\leq i \leq k$. 
We stress that this new amplification reduces an $n$-dimensional problem to a $(k+1)$-dimensional one, thus producing savings when $n$ is large relative to $k$. This is in contrast to methods I and II, where the setting remains in $n$ dimensions after amplification. The threshold of $p^{1/(k-1)+\kappa}$ as well as the conditions on $n, k$, and the dimension of $\sing F_k$ in Theorem \ref{thm:main} derive from the estimation of multi-dimensional exponential sums when bounding $S_1$. Likewise, the threshold of $p^{2/(k-2)+\kappa}$ and the conditions on $n,k$ in Theorem \ref{thm:main_diagonal} derive from the average of one-dimensional exponential sums when bounding $S_1$. 

The sums $S_1$ and $S_2$ in \eqref{eqn:S1S2_current} take on different shapes compared to their analogues in \eqref{eqn:overview_tasks} and \eqref{eqn:overview_tasks2}; these require new techniques to estimate, which we now discuss.

\subsubsection{On bounding $S_1$}
We first apply orthogonality to rewrite $S_1$ in \eqref{eqn:S1S2_current} in terms of exponential sums of the form
    \begin{align}\label{eqn:method_expsum}
        \sum_{(\xbf,\ybf)\in (\Nbf-\Hbf,\Nbf+\Hbf]\times (0,K]^n} e_p(\sum_{i=0}^k a_i f_i(\xbf,\ybf)),
    \end{align}
where $0\leq a_i \leq p-1$.
This shifts the problem of estimating a multiplicative character sum to additive character sums, which is an interesting feature of this method that had not appeared previously in the estimation of short, (purely) multiplicative character sums. 

There is vast literature on estimating exponential sums, and most known bounds on short exponential sums in many variables derive broadly from two directions. One is multi-dimensional Weyl differencing. For sums of exponentials evaluated at a single polynomial, Weyl differencing was generalized in the works of e.g. Tartakovsky \cite{Tar35} for a general polynomial (but with small savings), Davenport \cite{Dav59,Dav62_cubic29,Dav63_cubic16} for cubic forms, Chowla and Davenport \cite{ChoDav61} for binary cubics, Birch and Davenport \cite{BirDav61} for polynomials whose leading forms are ``diagonalizable'' over the complex numbers, and Wooley \cite{Woo99_weyl,Woo00_weyl} for binary polynomials of arbitrary degree (with savings comparable to the one-dimensional case). For systems of polynomials, Weyl differencing was further developed in the seminal works of Birch \cite{Bir62} and then of Schmidt \cite{Sch85_birch}, who applied the circle method to count solutions to systems of forms. Both works established bounds on exponential sums of the shape
    \begin{align}\label{eqn:birch_sum}
        \sum_{\xbf\in P\mathscr{B}} e(\sum_{i=1}^R \alpha_i g_i(\xbf)),
    \end{align}
    where $P$ is large, $\mathscr{B}$ is an $n$-dimensional box with side-length at most one and
    the $g_i$'s are forms of the same degree. Bounds on \eqref{eqn:birch_sum} were given by Birch, subject to a geometric condition involving a ``singular locus'', which consists of $\xbf$'s for which
    \begin{align}\label{eqn:birch_locus}
        \mathrm{rank}_{i,j} ((\partial g_i / \partial X_j)(\xbf)) <R.
    \end{align}
    On the other hand, Schmidt provided bounds subject to an algebraic condition involving the ``$h$-invariant'' of a form (which is roughly the smallest number $h$ such that the form can be written as a sum of $h$ many products). 
    Following their work, various adaptations of Weyl differencing have been developed to study related questions involving exponential sums; this includes, for example, the works of 
    \cite{Bra14, Sch14, Bra15,Die15, BHB17, BP17_birch, Mye18_quads, Bra21,Yam25}.

Another standard approach is Vinogradov's method of extracting point-wise bounds on exponential sums from mean value estimates; this was one of the main consequences of Vinogradov's mean value theorem. 
In the one-dimensional case, the resolution of the main conjecture by \cite{Woo16,BDG16,Woo19} gives the best known bound on one-dimensional Weyl sums, as recorded in \cite{Bou16}, with savings of $1/k(k-1)$, where $k$ is the degree of the polynomial. For $k\geq 6$, this is considerably better than the savings of $1/2^{k-1}$ produced by Weyl differencing. In higher dimensions, Arhipov, Karacuba and \v Cubarikov \cite{AKC80_book} extended the classical method of Vinogradov to sums of a single polynomial in many variables. Later, Parsell \cite{Par05} proved bounds on the number of solutions to a generalization of the Vinogradov system, which was subsequently improved by Parsell, Prendiville, and Wooley in \cite{PPW13}, with the resolution of the analogous main conjecture shown by  Guo and Zhang \cite{GZ19}. In \cite{PPW13}, the authors deduce bounds on a multi-dimensional Weyl sum associated to the Parsell-Vinogradov system. We note, however, that for an $n$-dimensional sum over a box with side-length $H$, this method saves only a small power of $H$ instead of a small power of $H^n$. Hence we will only apply this method when $n=1$.

We now return our focus back to the sum \eqref{eqn:method_expsum}, which resembles \eqref{eqn:birch_sum} in Birch's work, suggesting an implementation of his framework. But first, let us highlight four differences between our setting and that of Birch. To begin with, our system consists of polynomials that are not necessarily homogeneous (though this is not a problem for Weyl differencing), and our variables $\xbf,\ybf$ lie in boxes of different side-lengths; this second feature persists and requires additional work. The third and central complication arises from the condition on the singular locus; a direct adaptation of \cite{Bir62} generates a condition on the set of points for which \eqref{eqn:birch_locus} holds, with $g_i$ replaced by $f_i$ from \eqref{eqn:method_taylor}. Without an explicit description of this set, we would be forced to impose a rather convoluted assumption on the polynomials $F$ to which our theorem applies.
To circumvent this, we make the observation that \eqref{eqn:method_expsum} is related to the problem of counting rational lines on hypersurfaces, since by definition of the $f_i$'s,
    \begin{align*}
        \int_{0}^1\cdots \int_{0}^1\sum_{\xbf,\ybf} e(\sum_{i=0}^k \alpha_i f_i)d\alpha_0\cdots d\alpha_k = |\{(\xbf,\ybf):F(\xbf+t\ybf)=0\}|,
    \end{align*}
where $F(\xbf+t\ybf)=0$ as a polynomial in $t$. Counting linear spaces on hypersurfaces goes back to the work of Brauer \cite{Bra45} and Birch \cite{Bir57} and, more recently, for example, studied in \cite{Woo97,Par00, Par09_lines, Die10, Bra14, Bra15, BD21,Bra21, Woo24}, using various methods. The one most relevant to us (and we thank Damaris Schindler for pointing this out), for a general polynomial $F$, is that of Brandes in \cite{Bra14,Bra15}, where she applies Birch's method to the $f_i$'s each written as a symmetric, multilinear form (and adapts the method to boxes of different side-lengths). This effectively capitalizes on the relation of the $f_i$'s to the original polynomial $F$, and as Brandes remarks in \cite{Bra14}, this is considerably more efficient than the simpler, standard approach of treating the $f_i$'s as if they were genuinely distinct. In particular, this reduces the awkward condition on the singular locus of the system of polynomials $f_i$ to a condition on the more natural singular locus of $F$. Thus in the proof of Theorem \ref{thm:main}, where a priori we do not assume any structure on $F$, we apply Weyl differencing, following the work of Brandes, to bound $S_1$.

The fourth difficulty comes from a difference in perspectives, in terms of which parameters are fixed. In Birch's work (and, in general, works that apply the circle method), one typically fixes a polynomial at the start and studies its behaviour for varying primes. As a result, the coefficients of this fixed polynomial are viewed as constant. In our setting, this is reversed. We fix a modulus $p$ and hope to show nontrivial results for many polynomials, preferably uniformly over the coefficients. Thus, when we try to apply methods from the former setting, we encounter conditions that implicitly depend on the coefficients of $F$. (For example, we see this dependence in the rational approximation of the $\alpha_j$'s, say, in \cite[Lemma 2.5]{Bir62}.) To get around this, we apply ideas from Schmidt \cite{Sch84_expsums} who proved bounds on exponential sums that depend only on the number of variables and the degree (and we thank Rainer Dietmann for suggesting this reference).

Lastly, we make two remarks. First, as is standard in Weyl differencing, the savings are in the shape of the reciprocal of an exponential, and in the multi-dimensional case, this is reflected in the condition that $\smash{n-\dim  \sing F_k}$ is at least exponential in $k$. To weaken this, we can restrict $F$ to a class of diagonal polynomials, in which case the auxiliary polynomials $f_i$ are in two variables, so that fixing one variable allows us to apply the superior one-dimensional bounds produced by Vinogradov's mean value theorem. Thus in the proof of Theorem \ref{thm:main_diagonal}, we take this alternate approach to bound $S_1$.
Our second remark is that it should be possible to establish an analogue of Theorem \ref{thm:main}, with a condition that the $h$-invariant of $F$, in place of $\smash{n-\dim \sing F_k}$, is large relative to $k$. (Recently, Bailey and Lampert \cite{BaiLam24} proved that these two quantities are roughly equivalent, and we thank again Damaris Schindler for making us aware of this work.) 

\subsubsection{On bounding $S_2$}
Let us first briefly describe the shapes of $S_2$ in previous literature.
In the one-dimensional setting, we have
    \begin{align}\label{eqn:S2_classic}
        S_2 = \sum_{z\modd p} |\sum_{t\in (0,T]} \chi(t+z)|^{2r} = \sum_{\substack{t_1,...,t_{2r} \\ t_i\in (0,T]}}\sum_{z\modd p} \chi(\prod_{i=1}^r (z+t_i)\prod_{j=r+1}^{2r} (z+t_j)^{d-1}),
    \end{align}
    where $d$ is the order of $\chi$ and the second equality holds by expanding the $2r$-th moment and then swapping sums. One then splits the tuples $(t_1,...,t_{2r})$ into two types, ``good'' and ``bad'', where ``good'' are those such that the inner sum over $z$ can be bounded by the Weil bound (see e.g. \cite[Lemma 1]{Bur62a}) and ``bad'' are those such that the inner sum is bounded trivially, to yield a final bound on $S_2$ of $T^{2r}p^{1/2}+T^r p$.
    In method I for higher dimensions, 
    both the $t_i$'s and $z$ are $n$-dimensional, and so after expanding the moment, $\chi$ is evaluated at the product $\prod_{i=1}^r F(z+t_i)\prod_{j=r+1}^{2r} F(z+t_j)^{d-1}$. This is a highly singular polynomial in the variable $z$, for which there is no suitable generalization of the Weil bound giving square-root cancellation. The key innovation of \cite{PX20} sorts the $(t_1,...,t_{2r})$'s into different levels of ``bad'' and then applies a stratification result of Xu \cite{Xu20} to get a bound roughly of the form $T^{2nr}p^{n/2}+T^{nr} p^n+T^{2nr}p^{n/2}\sum_{i=1}^{n-1}T^{-i(r-1)/(n-1)}p^{i/2}$ (see \cite[Lemma 6.1]{PX20}), where in particular, the first two terms are the natural $n$-dimensional version of the one-dimensional bound and the last sum consists of the intermediate terms. In method II, $S_2$ reduces to one-dimensional complete sums of the shape \eqref{eqn:S2_classic}, but with $z$ running through a finite field $\mathbb{F}_{p^m}$ for some $1\leq m \leq n$ (see \cite[\S 3.3.2]{Chu25}), and for this, one can apply the Weil bound for finite fields (see e.g. \cite[Theorem 11.23]{IK04}) and obtain a final bound of $T^{2r}p^{n/2}+T^r p^{n}$.
    
    In our current case of \eqref{eqn:S1S2_current}, $z$ is $(k+1)$-dimensional and $\chi$ is evaluated at the product $\prod_{i=1}^r (z_0+z_1t_i+\cdots + z_kt_i^k)\prod_{j=r+1}^{2r} (z_0+z_1t_j+\cdots + z_kt_j^k)^{d-1}$. Since this is no longer a one-dimensional sum in $z$, we cannot apply the Weil bound as done in both the classic setting and method II. On the other hand, this polynomial is also not of the form treated by Pierce and Xu, where $z$ is shifted by the $t_i$'s, so we cannot utilize the stratification of Xu. To proceed, we must approach with new ideas.
    
    We shall exploit the unique structure of the product where each factor is linear in $z_0,...,z_k$ whose coefficients are increasing powers of the $t_i$'s (which together form a Vandermonde matrix). We show that when the parameter $r$ is small (relative to $k$), the inner sum over $z_0,...,z_k$ is zero for all ``good'' tuples $(t_1,...,t_{2r})$, which, roughly, are those such that the associated hyperplanes are linearly independent, and hence only the ``bad'' tuples contribute to the size of $S_2$, producing a final bound of $T^rp^{k+1}$. In particular, we do not apply any form of the Weil bound nor stratification results. Our bound is obtained at the cost of restricting $r$ to be small, but we will see that the bound on $S(F;\Nbf,\Hbf)$ is independent of $r$ (due to a cancellation from the bound on $S_1$). These features of our method are notably different from all previous techniques to bound $S_2$, which apply a square-root cancellation (or related) result to the ``good" tuples and which require $r$ to be large relative to $n$, generating a family of bounds on $S(F;\Nbf,\Hbf)$, where taking $r$ to infinity gives the threshold of $p^{\beta_n+\kappa}$ in method I and $ p^{1/4+\kappa}$ in method II.

\begin{remark}
    We note that the first step in method II of shifting by a product is also known in literature as the ``shift-by-$ab$'' trick which traces back to ideas of Vinogradov e.g. in \cite{Vin37,Vin38,Vin58} and, as mentioned, Karacuba. This technique has been extended to the estimation of other types of sums in analytic number theory; for example, this was applied in \cite{FI85_Kloosterman} for averages of Kloosterman sums, in \cite{FM98} for exponential sums over primes, and in \cite{bfkmm17,KMS17,KMS20} for bilinear forms in Kloosterman sums.
\end{remark}

\subsection{Outline of the paper}
In \S \ref{sec:burgess}, we carry out the amplification procedure and deduce Theorems \ref{thm:main} and \ref{thm:main_diagonal} conditional on the exponential sum bounds in Theorems \ref{thm:expsum_1} and \ref{thm:expsum_2}, respectively. In \S \ref{sec:expsum_1}, we prove Theorem \ref{thm:expsum_1} and in \S \ref{sec:expsum_2}, we prove Theorem \ref{thm:expsum_2}.

\subsection{Notation}
We use Vinogradov notation $A\ll B$ to denote $|A|\leq C|B|$ for some constant $C>0$ and $A\ll_\alpha B$ if the constant $C$ depends on $\alpha$. 
For a real number $\alpha$, let $e(\alpha)=e^{2\pi i \alpha}$ and let $\|\alpha\|$ denote the nearest distance to an integer. For a positive integer $q$ and $a\in \bZ/q\bZ$, let $e_q(a)=e^{2\pi i a/q}$.
For a finite set $S$, let $|S|$ denote its cardinality. 
We use boldface $\xbf\in \bR^m$ to denote the tuple $(x_1,...,x_m)$. For two integral vectors $\xbf$ and $\ybf$, we write $\xbf \equiv \ybf \modd p$ to mean $x_i\equiv y_i \modd p$ for all $i$. 
All interval notation represents discrete intervals, i.e. $[a,b]$ denotes all integers $x$ such that $a\leq x \leq b$.  For $\Nbf=(N_1,...,N_m)$ and $\Hbf=(H_1,...,H_m)$,
let $(\Nbf,\Nbf+\Hbf]$ denote the $m$-dimensional, discrete box $\prod_{i=1}^m (N_i,N_i+H_i]$, and let $\|\Hbf\| = \prod_{i=1}^m H_i$. 
For a vector $\xbf = (x_1,...,x_m)$ and a multi-index $\beta=(\beta_1,...,\beta_m)$, let $\xbf^\beta=x_1^{\beta_1}\cdots x_m^{\beta_m}$; furthermore, let $|\beta| = \beta_1+\cdots + \beta_m$ and $\beta! = \beta_1!\cdots \beta_m!$, and for a polynomial $f$ in $m$ variables, let $\partial^\beta f = \partial^{|\beta|}f/(\partial X_1^{\beta_1}\cdots \partial X_m^{\beta_m})$. Finally, we write $\dim$ to mean affine dimension over $\overline{\bF}_p$.

\section{Deduction of Theorems \ref{thm:main} and \ref{thm:main_diagonal} from bounds on exponential sums}\label{sec:burgess}

We begin by reducing the initial problem to estimating sum of characters over boxes all of whose side-lengths are comparable in size. This follows an idea from \cite[\S 3]{Kon10}, and we refer the reader to e.g. \cite[Lemma 3.1]{Chu25} for a proof.
\begin{lemma}\label{lem:comparable}
    Let $\Hbf=(H_1,...,H_n)$ and fix $H'$ such that $H'\leq H_i$ for all $i$. Then  \begin{align*}
       | S(F;\Nbf,\Hbf)| \leq 
       \|\Hbf\|
       \max_{\substack{\Nbf',\Hbf'\in \bZ^n\\ H'_i\in [H',2H']}} \|\Hbf'\|^{-1} |S(F;\Nbf',\Hbf')|.
    \end{align*}
\end{lemma}

Henceforth we assume that $\Hbf=(H_1,...,H_n)$ with $H\leq H_i\leq 2H$ for some $H>0$.

\subsection{The amplification method}\label{sec:amplification}
Fix $T,K\geq 1$ that we will choose later. Assume $TK=H$, and let $t\in (0,T]$ and $\ybf=(y_1,...,y_n)$ with $y_i\in (0,K]$. To carry out the first amplification steps, we adapt a technique of Friedlander and Iwaniec \cite{FI93}, who gave an alternate proof of the classical Burgess bound. In particular, they inserted a weight function that eliminates the error term from shifting, and then applied Fourier inversion. Here, we perform a multi-dimensional version of their strategy and remark that this has been applied previously, for example, in \cite{Ker14}.

For each $1\leq i \leq n$, define the weight $w_i(x_i) = \min\{x_i-N_i,1,N_i+H_i-x_i\}$ for $x_i\in (N_i,N_i+H_i]$ and $w_i(x_i)=0$ elsewhere. Define $w(\xbf) = w_1(x_1)\cdots w_n(x_n)$. Then
    \begin{align*}
            S(F;\Nbf,\Hbf) = T^{-1}K^{-n}\sum_{\xbf\in (\Nbf-\Hbf,\Nbf+\Hbf]}  \sum_{\ybf\in (0,K]^n} \sum_{t\in (0,T]} \chi(F(\xbf+t\ybf))w(\xbf+t\ybf),
    \end{align*}
since for each fixed $t,\ybf$, the box $(\Nbf-\Hbf + t\ybf, \Nbf+\Hbf + t\ybf]$ contains one copy of $(\Nbf-\Hbf,\Nbf+\Hbf]$.
By Fourier inversion,
    \begin{align*}
            S(F;\Nbf,\Hbf) = T^{-1}K^{-n}\sum_{\xbf\in (\Nbf-\Hbf,\Nbf+\Hbf]}  \sum_{\ybf\in (0,K]^n} \sum_{t\in (0,T]} \chi(F(\xbf+t\ybf))\int_{\bR^n} \hat{w}(\ubf) e(\ubf\cdot (\xbf+t\ybf)) d\ubf.
    \end{align*}
Let $v_i = u_iy_i$, so that the integral is
    \begin{align*}
         \int_{\bR^n} \hat{w}(\frac{v_1}{y_1},...,\frac{v_n}{y_n})y_1^{-1}\cdots y_n^{-1} e(\frac{v_1}{y_1}x_1+\cdots + \frac{v_n}{y_n}x_n)e(t\vbf\cdot \mathbf{1}) d\vbf,
    \end{align*}
    where $\mathbf{1}=(1,...,1)$.
 Then
       \begin{multline*}
         |   S(F;\Nbf,\Hbf) |\leq T^{-1}K^{-n}
          \int_{\bR^n} |\hat{w}(\frac{v_1}{y_1},...,\frac{v_n}{y_n})|\cdot|y_1^{-1}\cdots y_n^{-1}|\\
          \cdot \sum_{\xbf\in (\Nbf-\Hbf,\Nbf+\Hbf]}  \sum_{\ybf\in (0,K]^n}
           |\sum_{t\in (0,T]} \chi(F(\xbf+t\ybf)) e(t\vbf\cdot \mathbf{1})| d\vbf.
    \end{multline*}
By partial summation, $|\hat{w}_i(v_i/y_i)y_i^{-1}|\leq\min\{H_i,|v_i|^{-1}, K v_i^{-2}\}$, so that
upon removing the sum from the integral by inserting a supremum, we have
    \begin{multline*}
        |S(F;\Nbf,\Hbf)|\leq  T^{-1}K^{-n} 
          \int_{\bR^n}\prod_{i=1}^n \min\{H_i,|v_i|^{-1}, K v_i^{-2}\} d\vbf\\
        \cdot \sup_{\vbf\in [0,1]^n} \sum_{\xbf\in (\Nbf-\Hbf,\Nbf+\Hbf]}  \sum_{\ybf\in (0,K]^n}|\sum_{t\in (0,T]} \chi(F(\xbf+t\ybf)) e(t\vbf\cdot \mathbf{1})|.
    \end{multline*}
It is straightforward to check that the integral is bounded above by $(\log p)^n$, so that
    \begin{align*}
         |S(F;\Nbf,\Hbf)|\leq
         (\log p)^nT^{-1}K^{-n} \sum_{\xbf\in (\Nbf-\Hbf,\Nbf+\Hbf]}  \sum_{\ybf\in (0,K]^n}|\sum_{t\in (0,T]} \chi(F(\xbf+t\ybf)) e(t\vbf^\ast\cdot \mathbf{1})|,
    \end{align*}
    where $\vbf^\ast$ denotes the $\vbf\in [0,1]^n$ that achieves the supremum.
By Taylor expansion, we rewrite $F$  as a polynomial in $t$:
    \begin{align*}
        F(\xbf+t\ybf) 
        = \sum_{0\leq |\beta|\leq k} \frac{(\partial^\beta F)(\xbf)}{\beta!} (t\ybf)^\beta
        =\sum_{0\leq i \leq k} t^i 
        f_i(\xbf,\ybf),
    \end{align*}
where the coefficient of $t^i$ is
 \begin{align}\label{eqn:def_fi}
        f_i(\xbf,\ybf) = \sum_{ |\beta|=i}\frac{(\partial^\beta F)(\xbf)}{\beta!} \ybf^\beta.
    \end{align}
Then
\begin{align*}
    |S(F;\Nbf,\Hbf)| 
        &\leq  (\log p)^nT^{-1}K^{-n} \sum_{\xbf\in (\Nbf-\Hbf,\Nbf+\Hbf]}  \sum_{\ybf\in (0,K]^n}|\sum_{t\in (0,T]} \chi(\sum_{0\leq i \leq k} t^i 
        f_i(\xbf,\ybf)) e(t\vbf^\ast\cdot \mathbf{1})|.
\end{align*}
To remove dependence of the inner sum on $\xbf,\ybf$, we introduce a new variable $z_i=f_i(\xbf,\ybf)$ and define
\begin{align*}
    \eta(z_0,...,z_k)=|\{(\xbf,\ybf)\in (\Nbf-\Hbf,\Nbf+\Hbf]\times (0,K]^n: z_i\equiv f_i(\xbf,\ybf) \modd p \text{ for } 0\leq i \leq k\}|.
\end{align*}
Then
	\begin{align*}
	  |S(F;\Nbf,\Hbf)|  &\leq  (\log p)^nT^{-1}K^{-n} \sum_{\substack{z_0,...,z_k \\ \modd p}}\eta(z_0,...,z_k)|\sum_{t\in (0,T]} \chi(\sum_{0\leq i \leq k} t^i 
        z_i) e(t\vbf^\ast\cdot \mathbf{1})|.
	\end{align*}
Let $r\geq 1$. By H\"{o}lder's inequality, 
    \begin{multline*}
      | S(F;\Nbf,\Hbf)|\leq  (\log p)^n T^{-1}K^{-n}(\sum_{\substack{z_0,...,z_k\\ \modd p}}\eta(z_0,...,z_k) )^{1-1/r}(\sum_{\substack{z_0,...,z_k\\ \modd p}}\eta(z_0,...,z_k)^2 )^{1/2r}\\\cdot(\sum_{\substack{z_0,...,z_k\\ \modd p}}|\sum_{t\in (0,T]}\chi(z_0 + z_1t + \cdots + z_k t^k)e(t\vbf^\ast\cdot \mathbf{1})|^{2r})^{1/2r}.
    \end{multline*}
Denote these three sums respectively by $S_0,S_1,S_2$. We have the trivial estimate 
$ S_0 \leq 2^nK^n\|\Hbf\| $,
    so that
    \begin{align}\label{eqn:S_bound_diagonal}
       |  S(F;\Nbf,\Hbf)| \leq (\log p)^n T^{-1}K^{-n} (2^nK^n\|\Hbf\|)^{1-1/r}S_1^{1/2r}  S_2^{1/2r}.
    \end{align}
We will first bound $S_1$ conditionally, for the two classes of polynomials considered in Theorems \ref{thm:main} and \ref{thm:main_diagonal}, in parallel, and then bound $S_2$.

\subsection{Bounding $S_1$ conditionally}
By definition of $\eta(z_0,...,z_k)$, we have
\begin{align*}
    S_1=|\{(\xbf,\xbf',\ybf,\ybf')\in (\Nbf-\Hbf,\Nbf+\Hbf]^2\times (0,K]^{2n}:  f_i(\xbf,\ybf) \equiv  f_i(\xbf',\ybf') \modd p, 0\leq i \leq k\}|,
\end{align*}
where recall the definition of $f_i$ from $\eqref{eqn:def_fi}$.
 By orthogonality, we may rewrite $S_1$ in terms of exponential sums:
    \begin{align}
        S_1 
        &=p^{-(k+1)}\sum_{(\xbf,\xbf',\ybf,\ybf')\in (\Nbf-\Hbf,\Nbf+\Hbf]^2\times (0,K]^{2n}} \sum_{\substack{(a_0,...,a_k)\\ \modd p^{k+1}}}e_p(\sum_{i=0}^k a_i(f_i(\xbf,\ybf)-f_i(\xbf',\ybf'))) \nonumber \\
        &\ll_n H^{2n} K^{2n}p^{-(k+1)}+p^{-(k+1)}\sum_{\substack{(a_0,...,a_k)\not\equiv \mathbf{0}\\ \modd p^{k+1}}}|\sum_{(\xbf,\ybf)\in (\Nbf-\Hbf,\Nbf+\Hbf]\times (0,K]^n} e_p(\sum_{i=0}^k a_i f_i(\xbf,\ybf))|^2, \label{eqn:S1}
    \end{align}
where the first term in the second line is the contribution from $\abf \equiv \mathbf{0}$. So we focus on bounding the second term, where, for $\abf \not\equiv \mathbf{0}$, we denote the inner sum by $S(\abf;F)$.

\begin{thm}\label{thm:expsum_1}
 Let $n,k$ be positive integers with $n\geq k-1$. Let $F\in \bF_p[X_1,...,X_n]$ be a polynomial of degree $k$ with leading form $F_k$. 
   Let
$\Hbf=(H_1,...,H_n)$ with $H\leq H_i\leq 2H$ for some $H>0$, and suppose $p^{1/k}T^{1-1/k}< H \leq  p$, where $TK=H$ and $T,K\geq 1$.  Let 
    \begin{align}\label{eqn:cond_theta1}
       0<\theta <
       \frac{k - \log(p T^{k-1}) /  \log H}{k-1}.
    \end{align}
   Then for all $\varepsilon>0$,
        \begin{align*}
            S(\abf;F)
            \ll_{n,\varepsilon} H^nK^nH^{-(n-\dim \sing F_k)2^{-(k-1)}\theta + \varepsilon},
        \end{align*}
      where $\sing F_k$ denotes the singular locus of the hypersurface defined by $F_k=0$ over $\bF_p$. 
\end{thm}

The restriction $H> p^{1/k}T^{1-1/k}$ ensures that the upper bound in \eqref{eqn:cond_theta1} is positive; on the other hand, for $H$ around the size of $p$, the upper bound is $1-(\log T/\log p)$.
We prove Theorem \ref{thm:expsum_1} in \S \ref{sec:expsum_1}.
Assuming its truth for now, plug the bound into \eqref{eqn:S1} to get
    \begin{align*}
        S_1 
        \ll_{n,\varepsilon} H^{2n+\varepsilon}K^{2n}p^{-(k+1)} +H^{2n+\varepsilon}K^{2n} H^{-(n-\dim \sing F_k)2^{-(k-2)}\theta},
    \end{align*}
where the first term dominates as long as $  H\gg p^{\frac{2^{k-2}(k+1)}{\theta(n-\dim \sing F_k)}}.$ 
Rearranging gives
    \begin{align*}
        \theta \geq \frac{2^{k-2}(k+1)\log p}{(n-\dim \sing F_k)\log H},
    \end{align*}
which is compatible with \eqref{eqn:cond_theta1} as long as
    \begin{align}\label{eqn:cond_dimsing}
        n-\dim \sing F_k \geq 2^{k-2}(k^2-1) \frac{\log p}{\log (H^kp^{-1}T^{1-k})}.
    \end{align}
    Note that for any $c>0$, $\log p /\log (H^kp^{-1}T^{1-k}) \leq c$ if and only if 
        \begin{align}\label{eqn:cond_H_c}
            H\geq p^{(1/k)(1 + 1/c)}T^{1-1/k},
        \end{align}
        which covers the range $H\geq p^{1/k}T^{1-1/k}$ in the statement of Theorem \ref{thm:expsum_1}.
        Hence, given \eqref{eqn:cond_H_c}, \eqref{eqn:cond_dimsing} is satisfied as long as $ n-\dim \sing F_k \geq 2^{k-2}(k+1)(k-1)c $. For $c>1/2^{k-2}(k-1)$, this condition subsumes the condition $n\geq k-1$ from Theorem \ref{thm:expsum_1}. 
We summarize this as the following.
\begin{cor}\label{cor:expsum_1}
     Assume the hypotheses of Theorem \ref{thm:expsum_1}. 
     Fix $c>0$.
     Let $ n-\dim\sing F_k \geq 2^{k-2}(k^2-1)c$ and $H \geq p^{(1/k)(1+1/c)}T^{1-1/k}$. Then
 for all $\varepsilon>0$, $S_1\ll_{n,\varepsilon} H^{2n}K^{2n}p^{-(k+1)+\varepsilon}$.
\end{cor}

Now let us estimate $S_1$ when $F$ takes on the special shape of $F=G_1(X_1)+\cdots + G_{n}(X_{n})$. 
Since $F$ is diagonal, we can further split $S(\abf;F)$ into two-dimensional sums. For each of these, we fix one of the variables and apply a one-dimensional exponential sum bound (from Vinogradov's mean value theorem) to the remaining variable. Since the degree changes as $\abf$ varies, we obtain different bounds on $S(\abf;F)$ for different $\abf$. So, in contrast to a uniform pointwise bound in Theorem \ref{thm:expsum_1}, here we give a bound on the average of  $S(\abf;F)$.

\begin{thm}\label{thm:expsum_2}
Let $n\geq 1$ and $k\geq 5$. Let $F=G_1(X_1)+\cdots + G_{n}(X_{n})$, where each $G_j$ has degree $k$.  Let
$\Hbf=(H_1,...,H_n)$ with $H\leq H_i\leq 2H$ for some $H>0$, and suppose $p^{2/(k-2)}\ll K \leq H < p$.
Then for all $\varepsilon	>0$,
    \begin{align*}
        p^{-(k+1)}\sum_{\substack{(a_0,...,a_{k})\not\equiv \mathbf{0}\\ \modd p^{k+1}}}|S(\abf;F)|^2 \ll_{n,k,\varepsilon} H^{2n}
        K^{2n(1- \sigma(k))+\varepsilon},
    \end{align*}
where $\sigma(k) = 1/k(k-1)$.
\end{thm}

We prove this in \S \ref{sec:expsum_2}. Assuming its truth for now, plug the bound into \eqref{eqn:S1} to get
    \begin{align*}
        S_1\ll_{n,k,\varepsilon} H^{2n}K^{2n}p^{-(k+1)} + H^{2n}K^{2n-2n\sigma(k)+\varepsilon},
    \end{align*}
    where
the first term dominates when $K\gg p^{(k+1)/2n \sigma(k)}$. This is compatible with the condition $K\gg p^{2/(k-2)}$ as long as $n\sigma(k)\geq (k-2)(k+1)/4$.
We summarize this as the following.

\begin{cor}\label{cor:expsum_2}
Assume the hypotheses of Theorem \ref{thm:expsum_2}. Suppose $n\geq (k+1)k(k-1)(k-2)/4$, and let $p^{2/(k-2)}\ll K \leq H < p$. Then for all $\varepsilon	>0$, $S_1\ll_{n,k,\varepsilon} H^{2n}K^{2n}p^{-(k+1)+\varepsilon}$.
\end{cor}

This completes the two estimates of $S_1$, conditional on Theorems \ref{thm:expsum_1} and \ref{thm:expsum_2}.

\subsection{Bounding $S_2$}
Expanding the definition of $S_2$, we get
    \begin{align}
        S_2       
    &=  \sum_{\substack{z_0,...,z_k \\ z_i \modd p}} |\sum_{t\in (0,T]} \chi(z_0+z_1t+\cdots +z_kt^k) e(t\vbf^\ast \cdot \mathbf{1})|^{2r} \nonumber \\
    &=\sum_{\substack{t_1,...,t_{2r} \\ t_i\in (0,T]}}
        e((\sum_{i=1}^r t_i - \sum_{j=r+1}^{2r}t_j)\vbf^\ast \cdot \mathbf{1})
        \sum_{\substack{z_0,...,z_k\\ z_i \modd p}} \prod_{i=1}^{r}\chi(L(t_i;\zbf)) \prod_{j=r+1}^{2r}\overline{\chi}(L(t_j;\zbf))\nonumber \\
        &\leq \sum_{\substack{t_1,...,t_{2r} \\ t_i\in (0,T]}}
       |
        \sum_{\substack{z_0,...,z_k\\ z_i \modd p}} \chi(\prod_{i=1}^{r} L(t_i;\zbf) \prod_{j=r+1}^{2r} L(t_j;\zbf)^{d-1})|, \label{eqn:S2}
    \end{align}
where $d$ is the order of $\chi$ and $L(t;\zbf) = z_0 + z_1t + \cdots + z_k t^k$ is linear in the $z_i$'s.
We show the following bound.

\begin{prop}\label{prop:S2} Let $r\leq (k+1)/2$ be an integer. Then 
    $S_2\ll_r T^r p^{k+1}$.
\end{prop}

To prove this, note that each factor in the argument of $\chi$ defines a hyperplane with coefficients $1,t_i,...,t_i^k$. If these hyperplanes (say, after grouping the same ones together and reducing the exponents mod $d$) are linearly independent, then given $2r\leq k+1$, there exists a $(k+1)\times (k+1)$ nonsingular linear change of variables over $\bF_p$ under which we can rewrite the $(k+1)$-dimensional complete sum as one whose arguments are single-variable monomials in the new variables. Then the sum vanishes by orthogonality. This motivates us to split the $(t_1,...,t_{2r})$'s into ``good'' and ``bad'', where ``good'' are those such that there is at least one linearly independent hyperplane. 

To determine whether the hyperplanes are linearly independent, we capitalize on the shape of the coefficients, which are increasing powers of $t_i$'s. This conveniently forms a Vandermonde matrix, for which we have the following standard fact.

\begin{lemma}\label{lem:vandermonde}
 Let $m,n$ be positive integers and  consider the Vandermonde matrix
    \begin{align*}
        \begin{bmatrix}
            1&a_1&a_1^2&\cdots&a_1^n\\
            \vdots &&&&\vdots\\
            1&a_m&a_m^2&\cdots&a_m^n
        \end{bmatrix}.
    \end{align*}
    If $m\leq n+1$, then this has rank $m$ if and only if all $a_i$'s are distinct. If $m>n+1$, then this has rank $n+1$ if and only if $n+1$ of the $a_i$'s are distinct.
\end{lemma}

\begin{proof}[Proof of Proposition \ref{prop:S2}]
For $\tbf \in (0,T]^{2r}$, let $S(\tbf)$ denote the inner $(k+1)$-dimensional sum in \eqref{eqn:S2}. Let $\mathrm{Bad}(T)$ denote the set of $\tbf=(t_1,...,t_{2r})$ such that for every $t_i$, there exists $j\neq i$ for which $L(t_i;\Zbf)=L(t_j;\Zbf)$ as polynomials in $\Zbf$. Note that $L(t_i;\Zbf)=L(t_j;\Zbf)$ if and only if $t_i=t_j$. Let $\mathrm{Good}(T)$ denote the complement of $\mathrm{Bad}(T)$. Then
    \begin{align}\label{eqn:S2_goodbad}
        S_2 \leq \sum_{\tbf \in \mathrm{Bad}(T)} |S(\tbf)| + \sum_{\tbf \in \mathrm{Good}(T)} |S(\tbf)|.
    \end{align}
We claim that $|\mathrm{Bad}(T)|\ll_r T^r$. This argument is standard, and we include it for completion. Let $\tbf\in \mathrm{Bad}(T)$ and let $m$ denote the number of distinct $t_i$'s in $\tbf$, so that $1\leq m \leq r$ by definition of $\mathrm{Bad}(T)$. With $m$ fixed, let $x_1,...,x_m$ denote the distinct values.
For each $1\leq j \leq m$, let $i_j$ denote the smallest index $i$ such that $t_i=x_j$; suppose without loss of generality that $i_1<\cdots < i_m$.
Note that $i_1=1$ and $i_m\leq 2r-1$. The number of choices of these positions $i_1,...,i_m$ is $\binom{2r-2}{m-1} \leq 2^{2r-2}$. 
With a position $i_1,...,i_m$ fixed, there are $T^m$ choices for $t_{i_1},...,t_{i_m}$, while the remaining $2r-m$ positions can take on any of these $m$ values, for which there are $m^{2r-m}$ choices. Hence $|\mathrm{Bad}(T)|$ is bounded above by $ \sum_{m=1}^r 2^{2r-2}T^m m^{2r-m} \leq 2^{2r-2}T^r r^{2r} r \ll_r T^r$, which proves the claim.
For $\tbf\in \mathrm{Bad}(T)$, we bound $S(\tbf)$ trivially by $p^{k+1}$.

For each $\tbf\in \mathrm{Good}(T)$, we claim that $S(\tbf)=0$. Fix $\tbf \in \mathrm{Good}(T)$, so that by definition, there exists $t_i$ such that $t_i\neq t_j$ for all $j\neq i$. Suppose without loss of generality that this is $t_1$. By grouping the same $t_j$'s together, rewrite $S(\tbf)$ as
    \begin{align*}
    S(\tbf)=     \sum_{\substack{z_0,...,z_k\\ z_i \modd p}} \chi(L(t_1;\zbf)L(t_2;\zbf)^{\alpha_2}\cdots L(t_m;\zbf)^{\alpha_m}),
    \end{align*}
where the $t_i$'s are distinct, $1\leq m\leq 2r \leq k+1$ by assumption, and without loss of generality, $\alpha_i\geq 1$ are integers not divisible by $d$. By Lemma \ref{lem:vandermonde}, the $m\times (k+1)$ matrix of coefficients 
    \begin{align*}
        \begin{bmatrix}
             1&t_1&t_1^2&\cdots&t_1^k\\
            \vdots &&&&\vdots\\
            1&t_m&t_m^2&\cdots&t_m^k
        \end{bmatrix}
    \end{align*}
has rank $m$, since $t_1,...,t_m$ are all distinct. Extend this to a nonsingular $(k+1)\times (k+1)$ matrix $A$, and consider the change of variables $\wbf = A\zbf$. Then
 \begin{align*}
    S(\tbf)=     \sum_{\substack{w_1,...,w_{k+1}\\ w_i \modd p}} \chi(w_1)\chi(w_2^{\alpha_2})\cdots \chi(w_m^{\alpha_m})
    =\sum_{\substack{w_2,...,w_{k+1}\\ w_i \modd p}}\chi(w_2^{\alpha_2})\cdots \chi(w_m^{\alpha_m}) \sum_{w_1 \modd p} \chi(w_1)=0,
    \end{align*} 
    so $\tbf\in \mathrm{Good}(T)$ does not contribute to the size of $S_2$. Thus \eqref{eqn:S2_goodbad} is $S_2\leq \sum_{\tbf \in \mathrm{Bad}(T)} |S(\tbf)| \ll_r T^rp^{k+1}$ as desired.
\end{proof}

\subsection{Deduction of Theorems \ref{thm:main} and \ref{thm:main_diagonal} conditional on Corollaries \ref{cor:expsum_1} and \ref{cor:expsum_2}} 
We combine the bound on $S_2$ and the conditional bounds on $S_1$ to deduce the main theorems. Note that Corollaries \ref{cor:expsum_1} and \ref{cor:expsum_2} give the same bound on $S_1$, with different conditions on $F,H,K,n,k$.

Let $F$ be a polynomial of degree $k$ with leading form $F_k$. Let $c>0$, and suppose $n-\dim_{\overline{\bF}_p}\sing F_k\geq 2^{k-2} (k^2-1)c$ and $H\geq p^{(1/k)(1+1/c)}T^{1-1/k}$. Then by \eqref{eqn:S_bound_diagonal}, Corollary \ref{cor:expsum_1} and Proposition \ref{prop:S2}, for all $\varepsilon>0$, we have
    \begin{align*}
    S(F;\Nbf,\Hbf)
   \ll_{n,r,\varepsilon} T^{-1}K^{-n} (H^n K^n)^{1-1/r}(H^{2n}K^{2n}p^{-(k+1)+\varepsilon})^{1/2r}  (T^rp^{k+1})^{1/2r}
   \ll_{n,k,\varepsilon} H^n T^{-1/2} p^{\varepsilon}.
        \end{align*}
Write $T=p^\tau$, so that assuming $H\gg p^{(1/k)(1+1/c)+\tau(1-1/k)}$ gives $S(F;\Nbf,\Hbf) \ll_{n,k,\varepsilon} H^np^{-\tau/2+\varepsilon}$. 
Finally, apply Lemma \ref{lem:comparable} with $H'=H_{\min}$ and choose $c=k-1$ to complete the deduction of Theorem \ref{thm:main}.

Now let $F=G_1(X_1)+\cdots + G_{n}(X_{n})$ and suppose $n>(k-2)(k-1)k(k+1)/4$ and $K\gg p^{2/(k-2)}$.
By \eqref{eqn:S_bound_diagonal}, Corollary \ref{cor:expsum_2} and Proposition \ref{prop:S2}, again we get 
    $S(F;\Nbf,\Hbf) \ll_{n,k,\varepsilon} H^{n}T^{-1/2}p^{\varepsilon}$
for all $\varepsilon>0$.
    Writing $TK=H$ and $T=p^\tau$, we have the condition $H\gg p^{2/(k-2) + \tau}$ and the bound $S(F;\Nbf,\Hbf) \ll_{n,k,\varepsilon} H^{n}p^{-\tau/2+\varepsilon}$. 
    As above, apply Lemma \ref{lem:comparable} with $H'=H_{\min}$ to complete the deduction of Theorem \ref{thm:main_diagonal}.

\section{Proof of Theorem \ref{thm:expsum_1}}
\label{sec:expsum_1}

We follow the method of Birch as adapted by Brandes in \cite{Bra14,Bra15} for linear spaces on hypersurfaces and for boxes of different size.
As mentioned in the introduction, this is relevant to our work since our exponential sum of interest \eqref{eqn:S1} is closely related to counting lines on $F=0$.
The key first step is to decompose $F$ as a sum of forms $F_j$ of different degrees and to rewrite each $F_j$ as a multilinear form, so that we can exploit the relation among the $f_i$'s (as partials of $F$). This will later substantially simplify the condition on the singular locus that appears at the end of the Weyl differencing process as treated by Birch.
There, one typically arrives at one of three possibilities (e.g. as in \cite[Lemma 2.3]{Bir62} or \cite[Lemma 3.5]{Bra15}): either the exponential sum has a nontrivial bound, or the coefficients in the sum have rational approximation with small denominator, or a certain system of equations has many solutions. Since we are ultimately concerned with rational coefficients, we deviate from the usual trichotomy and instead adapt ideas of \cite{Sch84_expsums} to give a variation where either the exponential sum is small or a certain system has many solutions mod $p$.

Let $F(\Xbf)\in \bZ[X_1,...,X_n]$ be of degree $k$.
Write $F(\Xbf) = F_k(\Xbf) + F_{k-1}(\Xbf) + \cdots + F_0(\Xbf)$, where $F_j$ is a form of degree $j$. We apply the following standard fact for homogeneous polynomials and refer readers to e.g. \cite[Chapter 3, \S 2.2]{Pro07} for a proof.

\begin{lemma} Let
$P\in \bZ[X_1,...,X_n]$ be a homogeneous polynomial of degree $d$. Then there exists a symmetric, multilinear polynomial $\Phi(\Xbf_1,...,\Xbf_d)$ such that $P(\Xbf) = \Phi(\Xbf,...,\Xbf)$.
\end{lemma}

Applying the lemma to each $F_j$, there exists a symmetric, multilinear form $\Phi_j(\Xbf_1,...,\Xbf_j)$ such that $F_j(\Xbf) = \Phi_j(\Xbf,...,\Xbf)$. With this identification, we apply linearity of $\Phi_j$ to expand $F_j(\Xbf + t\Ybf)$. To keep track of the number of $\Xbf$'s and $\Ybf$'s, we make the following definition: for indeterminantes $\Zbf_1,...,\Zbf_r$ and non-negative integers $z_1,...,z_r$ such that $z_1+\cdots + z_r=j$, let
    \begin{align}\label{def:tilde_phi}
        \tilde{\Phi}_{j,(z_1,...,z_r)}(\Zbf_1,...,\Zbf_r)
=\Phi_j(\Zbf_1,...,\Zbf_1,...,\Zbf_r,...,\Zbf_r),
    \end{align}
    where, on the right-hand side, each $\Zbf_i$ appears $z_i$ many times.
Then
    \begin{align*}
        F(\Xbf+t\Ybf) = \sum_{j=0}^k F_j(\Xbf+t\Ybf) = \sum_{j=0}^k \sum_{i=0}^j t^i \binom{j}{i} \Tilde{\Phi}_{j, (j-i,i)}(\Xbf,\Ybf)
        =\sum_{i=0}^k t^i \sum_{j=i}^k  \binom{j}{i} \Tilde{\Phi}_{j,(j-i,i)}(\Xbf,\Ybf).
    \end{align*}
In particular, note the property that, for $0\leq i \leq  j \leq k$,
    \begin{align}\label{eqn:property_partial_binom}
        \sum_{|\beta|=i} \frac{(\partial^\beta F_j)(\Xbf)}{\beta!}\Ybf^{\beta} = \binom{j}{i} \tilde{\Phi}_{j,(j-i,i)}(\Xbf,\Ybf).
    \end{align}
This relation allows us to write our system of forms $f_i$ as defined in \eqref{eqn:def_fi} in terms of the same multilinear form, so that we are really working with one polynomial rather than $k+1$ many. 

To keep track of variables more easily, we switch to the notation of $\Xbf_1, \Xbf_2$ in place of $\Xbf,\Ybf$.
For a real tuple $\alpha=(\alpha_0,...,a_{k})$, define
    \begin{align}\label{def:tildeF}
        \tilde{F}(\Xbf_1,\Xbf_2;\alpha) = \sum_{i=0}^k \alpha_i\sum_{j=i}^k  \binom{j}{i} \Tilde{\Phi}_{j,(j-i,i)}(\Xbf_1,\Xbf_2),
    \end{align}
    where, suppose $\alpha_I\not\in \bZ$ for some index $0 \leq I \leq k$ (though we will not apply this until Proposition \ref{prop:tripartite}),
and for fixed $\Mbf_1,\Mbf_2,\Lbf_1,\Lbf_2$, define
    \begin{align*}
        S(\alpha) = \sum_{\substack{\xbf_1, \xbf_2 \\\xbf_i\in (\Mbf_i, \Mbf_i+\Lbf_i]}}e( \tilde{F}(\xbf_1,\xbf_2;\alpha)).
    \end{align*}
Assume that for $i=1,2$, $\Lbf_i = (L_{i,1},...,L_{i,n})$ where $L_i\leq L_{i,j}\leq 2 L_i$ for all  $1\leq j \leq n$,  for some $L_i>0$ (though, again, we will not require this dyadic variation until Lemma \ref{lem:N1}.)

In \S \ref{sec:weyl}, we carry out Weyl differencing and in \S \ref{sec:expsum_large}, we examine what happens when the exponential sum is large.
Then in \S \ref{sec:pf_22}, we deduce Theorem \ref{thm:expsum_1}.

\subsection{Multi-dimensional Weyl differencing}\label{sec:weyl}
We begin with the standard initial step of introducing a differencing operator that reduces the degree of the polynomials by repeated applications of Cauchy-Schwarz. This produces an exponential sum that is linear in $\xbf_1,\xbf_2$ which we bound in Proposition \ref{lem:weyl2} in terms of the expressions $\alpha_I C(k)\Psi_m(\hbf_1,...,\hbf_{k-1})$, for $1\leq m \leq n$, where $C(k)$ is a constant and $\Psi_m$ is defined in \eqref{eqn:def_psi} as the coefficients of $x_m$ in $\tilde{\Phi}_{k,(1,1,...,1)}(\xbf,\hbf_1,...,\hbf_{k-1})$. Then we demonstrate in Proposition \ref{prop:weyl3} that either $S(\alpha)$ can be bounded nontrivially or these expressions $\alpha_I C(k)\Psi_m(\hbf_1,...,\hbf_{k-1})$, $1\leq m \leq n$, are close to integers for many choices of $\hbf_1,...,\hbf_{k-1}$. Finally, we show that the latter case implies that the dimension of the singular locus of $F_k$ is large, culminating in Proposition \ref{prop:tripartite}.

  For a polynomial $P(\Zbf_1,...,\Zbf_r)$, a tuple $\hbf=(h_1,...,h_n)$, and an integer $1\leq s \leq r$, define the differencing operator
    \begin{align*}
\Delta_{s,\hbf}P(\Zbf_1,...,\Zbf_r)=P(\Zbf_1,...,\Zbf_s + \hbf,...,\Zbf_r)-P(\Zbf_1,...,\Zbf_r).
    \end{align*}

\begin{lemma}\label{lem:weyl_differencing} Let $1\leq \kappa\leq k-1$. Then for any $(s_1,...,s_\kappa)\in \{1,2\}^\kappa$,
    \begin{align*}
     |S(\alpha)|^{2^\kappa}
        \ll_n
        (L_1^n L_2^n)^{2^\kappa -1}
        (\prod_{i=1}^\kappa L_{s_i}^n)^{-1} \sum_{\substack{\hbf_1,...,\hbf_\kappa \\ \hbf_i\in [-\Lbf_{s_i},\Lbf_{s_i}]}}\sum_{\substack{\xbf_1,\xbf_2 \\ \xbf_i \in B^{\kappa}_{i}}}e(\Delta_{s_\kappa,\hbf_\kappa}\cdots \Delta_{s_1,\hbf_1}\Tilde{F}(\xbf_1,\xbf_2;\alpha)),
    \end{align*}
    where $B^{\kappa}_{i}$ are boxes contained in $(\Mbf_i,\Mbf_i+\Lbf_i]$.
\end{lemma}
    \begin{proof}
        We prove by induction on $\kappa$. Let $\kappa=1$ and $s_1, s\in \{1,2\}$ such that $s\neq s_1$. Applying the Cauchy-Schwarz inequality and a change of variables,
        \begin{align*}
            |S(\alpha)|^{2}   
        &\leq (\sum_{\xbf_{s} \in (\Mbf_{s}, \Mbf_{s} + \Lbf_{s} ]} 1)(\sum_{\xbf_{s} \in (\Mbf_{s}, \Mbf_{s} + \Lbf_{s} ]}|\sum_{\xbf_{s_1} \in (\Mbf_{s_1}, \Mbf_{s_1} + \Lbf_{s_1} ]}   
        e( \tilde{F}(\xbf_1,\xbf_2;\alpha))|^2)\\
     &    \leq \|\Lbf_{s}\|\sum_{\xbf_{s}\in (\Mbf_{s}, \Mbf_{s} + \Lbf_{s} ]}\sum_{\substack{\hbf_1\in [-\Lbf_{s_1},\Lbf_{s_1}] }}\sum_{\xbf_{s_1}\in B^1_{s_1}}
        e(\Delta_{s_1,\hbf_1}\Tilde{F}(\xbf_1,\xbf_2;\alpha)),
        \end{align*}
    where $B^1_{s_1}=\{\xbf_{s_1}\in (\Mbf_{s_1},\Mbf_{s_1} + \Lbf_{s_1}]:\xbf_{s_1} + \hbf_1\in (\Mbf_{s_1},\Mbf_{s_1} + \Lbf_{s_1}]\}$.
This proves the inequality for $\kappa=1$.
Suppose now that the inequality holds for a given $\kappa\geq 1$. Let $(s_1,...,s_{\kappa+1})\in \{1,2\}^{\kappa+1}$ and let $s \in \{1,2\}$ such that $s \neq s_{\kappa+1}$. Then by the inductive hypothesis,
\begin{align*}
        |S(\alpha)|^{2^{\kappa+1}}\ll_n
        (L_1^nL_2^n)^{2^{\kappa+1} -2}
        (\prod_{i=1}^\kappa L_{s_i}^n)^{-2}\Sigma,
    \end{align*}
where
    \begin{align*}
         \Sigma &=|\sum_{\substack{\hbf_1,...,\hbf_{\kappa} \\ \hbf_i\in [-\Lbf_{s_i},\Lbf_{s_i}]}}\sum_{\substack{\xbf_1,\xbf_2 \\ \xbf_i\in B^{\kappa}_{i}}}e(\Delta_{s_\kappa,\hbf_{\kappa}}\cdots \Delta_{s_1,\hbf_1}\Tilde{F}(\xbf_1,\xbf_2;\alpha))|^2.
    \end{align*}
Again by Cauchy-Schwarz and a change of variables,
    \begin{multline*}
      | \Sigma|
        \leq(\prod_{i=1}^\kappa \|\Lbf_{s_i}\|) \|\Lbf_{s}\|
        \sum_{\substack{\hbf_1,...,\hbf_{\kappa} \\ \hbf_i\in [-\Lbf_{s_i},\Lbf_{s_i}]}}\sum_{\substack{\xbf_{s}\in B^{\kappa}_{s}}}
        \sum_{\hbf_{\kappa+1}\in [-\Lbf_{s_{\kappa+1}},\Lbf_{s_{\kappa+1}}]}\\
        \sum_{\xbf_{s_{\kappa+1}}\in B^{\kappa+1}_{s_{\kappa+1}}}
        e(\Delta_{s_{\kappa+1},\hbf_{\kappa+1}}\Delta_{s_\kappa,\hbf_{\kappa}}\cdots \Delta_{s_1,\hbf_1}\Tilde{F}(\xbf_1,\xbf_2;\alpha)),
    \end{multline*}
    where $B^{\kappa +1}_{s_{\kappa+1}} = \{\xbf_{s_{\kappa +1}}\in B^{\kappa}_{s_{\kappa+1}} : \xbf_{s_{\kappa+1}} + \hbf_{\kappa+1} \in B^{\kappa}_{s_{\kappa+1}}\}$, which completes the proof.
    \end{proof}

To evaluate the right-hand side of Lemma \ref{lem:weyl_differencing}, we now write out explicitly the output of the differencing operator. For ease of notation, let $\Delta^{\kappa}_{\underline{s},\underline{\hbf}} = \Delta_{s_\kappa,\hbf_{\kappa}}\cdots \Delta_{s_1,\hbf_1}$; additionally, in Lemma \ref{lem:delta_explicit} and Proposition \ref{lem:weyl2}, we use superscripts for the letters $\xi,\eta,\sigma$ to denote indices instead of exponents.

\begin{lemma}\label{lem:delta_explicit}
    Fix an integer $\kappa\geq 1$. Let $\underline{s}=(s_1,...,s_\kappa)\in \{1,2\}^\kappa$, and for $i=1, 2$, define 
        \begin{align*}
    \sigma_{\kappa,\underline{s}}^i = |\{s\in \{s_1,...,s_\kappa\}:s= i\}|.
        \end{align*}      
    Fix tuples $\hbf_1,...,\hbf_{\kappa}$ with $\hbf_i\in [-\Lbf_{s_i},\Lbf_{s_i}]$.  Let $1\leq j\leq k$ and $0\leq \xi_{0}^1,\xi_{0}^2\leq j$ be integers satisfying $\xi_{0}^1 + \xi_{0}^2=j$.  Then,  if $\xi_{0}^1 <\sigma_{\kappa,\underline{s}}^1$ or $\xi_{0}^2 <\sigma_{\kappa,\underline{s}}^2$, then $\Delta_{\underline{s},\underline{\hbf}}^{\kappa}\tilde{\Phi}_{j,(\xi^{1}_{0},\xi^{2}_{0})}(\xbf_1,\xbf_2)$ vanishes.
    Otherwise,
 upon recalling the definition \eqref{def:tilde_phi},
    \begin{align}\label{eqn:delta_explicit2}
\Delta_{\underline{s},\underline{\hbf}}^{\kappa}\tilde{\Phi}_{j,(\xi^{1}_0,\xi^{2}_0)}(\xbf_1,\xbf_2)
                = \sum_{\iota_\kappa=0}^{\xi_{\kappa-1}^{s_\kappa}-1}\cdots \sum_{\iota_1=0}^{\xi_{0}^{s_1}-1}\binom{\xi_{\kappa-1}^{s_\kappa}}{\iota_\kappa}\cdots \binom{\xi_{0}^{s_1}}{\iota_1} \tilde{\Phi}_{j,(\xi_\kappa^{1},\xi_\kappa^{2},\eta_\kappa^{1},...,\eta_\kappa^{\kappa})}(\xbf_1,\xbf_2,\hbf_1,...,\hbf_{\kappa}),
            \end{align}
        where
            \begin{align*}
                (\xi^1_{\kappa},\xi^2_{\kappa},\eta^1_{\kappa},...,\eta^{\kappa}_{\kappa})=
                \begin{cases}
                    (\xi^1_{\kappa-1},\iota_\kappa,\eta^1_{\kappa-1},...,\eta_{\kappa-1}^{\kappa-1},\xi_{\kappa-1}^{2}-\iota_\kappa), &s_\kappa=1\\
                     (\iota_\kappa,\xi_{\kappa-1}^{2},\eta_{\kappa-1}^{1},...,\eta_{\kappa-1}^{\kappa-1},\xi_{\kappa-1}^{1}-\iota_\kappa), &s_\kappa=2,
                \end{cases}
            \end{align*}
       and $\eta_\kappa^{i}\geq 1$ for all $1\leq i \leq \kappa-1$.
       In particular, $\Delta_{\underline{s},\underline{\hbf}}^{\kappa}\tilde{\Phi}_{j,(\xi^{1}_0,\xi^{2}_0)}(\xbf_1,\xbf_2)$ has bihomogeneous degree $(\xi^{1}_0-\sigma^{1}_{\kappa,\underline{s}},\xi^{2}_0-\sigma^{2}_{\kappa,\underline{s}})$ in the variables $(\xbf_1,\xbf_2)$.
\end{lemma}
    \begin{proof}
    We prove by induction on $\kappa$. Let $\kappa =1$, and suppose without loss of generality that $s_1=1$. If $\xi^{1}_0=0$, then $\tilde{\Phi}_{j,(\xi^{1}_0,\xi^{2}_0)}(\xbf_1,\xbf_2)$ is independent of $\xbf_1$, so for any $\hbf_1$, $\Delta_{s_1,\hbf_1}\tilde{\Phi}_{j,(\xi^{1}_0,\xi^{2}_0)}(\xbf_1,\xbf_2)=0$. Otherwise, by linearity,
        \begin{align*}
            \Delta_{s_1,\hbf_1}\tilde{\Phi}_{j,(\xi^{1}_0,\xi^{2}_0)}(\xbf_1,\xbf_2) 
            &= \tilde{\Phi}_{j,(\xi^{1}_0,\xi^{2}_0)}(\xbf_1+ \hbf_1,\xbf_2 ) - \tilde{\Phi}_{j,(\xi^{1}_0,\xi^{2}_0)}(\xbf_1,\xbf_2)\\
            &=\sum_{\iota_1=0}^{\xi^{1}_0-1} \binom{\xi^{1}_0}{\iota_1}\tilde{\Phi}_{j,(\iota_1,\xi^{2}_0,\xi^{1}_0-\iota_1)}(\xbf_1,\xbf_2,\hbf_1).
        \end{align*} 
    The term with the highest degree in $\xbf_1$ occurs when $\iota_1=\xi^{1}_0-1$, in which case the right-hand side has bihomogeneous degree $(\xi^{1}_0-1,\xi^{2}_0)$ in $(\xbf_1,\xbf_2)$.
    Additionally, $\eta_1^{1}=\xi^{1}_0-\iota_1\geq 1$. This proves the lemma for $\kappa=1$.

    Now assume the lemma holds for some $\kappa\geq 1$. Fix $\underline{s}'=(s_1,...,s_{\kappa+1})\in \{1,2\}^{\kappa+1}$, and let $\underline{s}=(s_1,...,s_{\kappa})$. By the inductive hypothesis, if $\xi^{1}_0<\sigma^{1}_{\kappa,\underline{s}}$ or $\xi^{2}_0<\sigma^{2}_{\kappa,\underline{s}}$, then for any $\hbf_1,...,\hbf_{\kappa}$,
        \begin{align}\label{eqn:delta_explicit3}
            \Delta^{\kappa}_{\underline{s},\underline{\hbf}}\tilde{\Phi}_{j,(\xi^{1}_0,\xi^{2}_0)}(\xbf_1,\xbf_2)
        \end{align}
    vanishes. So assume otherwise; then
           \eqref{eqn:delta_explicit3}
    has bihomogeneous degree $( \xi^{1}_0-\sigma^{1}_{\kappa,\underline{s}},\xi^{2}_0-\sigma^{2}_{\kappa,\underline{s}})$. Suppose without loss of generality that $s_{\kappa+1}=1$.
If $\xi^{1}_0<\sigma_{\kappa+1,\underline{s}'}^{1}$, we deduce that $\xi^{1}_0=\sigma^{1}_{\kappa,\underline{s}}$ and \eqref{eqn:delta_explicit3} is independent of $\xbf_1$, and hence 
    $\Delta_{s_{\kappa+1},\hbf_{\kappa+1}}$ applied to \eqref{eqn:delta_explicit3} vanishes. If instead $\xi^{1}_0\geq \sigma_{\kappa+1,\underline{s}'}^{1}$, then again by linearity, $\Delta_{s_{\kappa+1},\hbf_{\kappa+1}}\Delta^{\kappa}_{\underline{s},\underline{\hbf}}\tilde{\Phi}_{j,(\xi^{1}_0,\xi^{2}_0)}(\xbf_1,\xbf_2)$ is
    \begin{align*}
 \sum_{\iota_{\kappa+1}=0}^{\xi_\kappa^{s_{\kappa+1}}-1}\cdots \sum_{\iota_1=0}^{\xi_0^{s_1}-1}
        \binom{\xi_\kappa^{s_{\kappa+1}}}{\iota_{\kappa+1}}
        \cdots \binom{\xi_0^{s_1}}{\iota_1} 
        \tilde{\Phi}_{j,(\iota_{\kappa+1},\xi_\kappa^{2},\eta_\kappa^{1},...,\eta_\kappa^{\kappa}, \xi_\kappa^{1} - \iota_{\kappa+1})}(\xbf_1,\xbf_2,\hbf_1,...,\hbf_{\kappa}, \hbf_{\kappa+1}).
    \end{align*}
The term with the highest degree in $\xbf_1$ occurs when
    \begin{align*}
        \iota_{\kappa+1}=\xi_\kappa^{s_{\kappa+1}}-1=\xi_\kappa^{1}-1
        =\xi^{1}_0-\sigma^{1}_{\kappa,\underline{s}}-1 = \xi^{1}_0 - \sigma^{1}_{\kappa+1,\underline{s}'},
    \end{align*}
    while the degree in $\xbf_2$ remains at $\xi^{2}_0 - \sigma^{2}_{\kappa,\underline{s}}=\xi^{2}_0 - \sigma^{2}_{\kappa+1,\underline{s}'}$.
Therefore, the bihomogeneous degree is $( \xi^{1}_0 - \sigma^{1}_{\kappa+1,\underline{s}'}, \xi^{2}_0 - \sigma^{2}_{\kappa+1,\underline{s}'})$. Moreover, $\eta_{\kappa+1}^{\kappa} = \xi_\kappa^{1} - \iota_{\kappa+1} \geq 1$.
This completes the proof of the lemma.
    \end{proof}

Now we combine the previous two lemmas. Recall that we assumed $\alpha_I\not\in\bZ$, so we focus on the special case $\underline{s}=(s_1,...,s_{k-1})$, where $s_i=1$ for $1\leq i \leq k-(I+1)$ and $s_i=2$ for $k-I\leq i \leq k-1$. (We remark that this assumption is not mandatory, but it simplifies the steps notationally.)

In this following proposition, we shall see the effect of writing each $F_j$ as a symmetric, multilinear form. Our original system of polynomials $f_i(\xbf_1,\xbf_2)$ reduces to the system of forms $\tilde{\Phi}_{k,(k-i,i)}(\xbf_1,\xbf_2)$ for $0\leq i \leq k$. With the above choice of $\underline{s}$, for any $\underline{\hbf}$, the differencing operator $\Delta^{k-1}_{\underline{s},\underline{\hbf}}$ annihilates $\tilde{\Phi}_{k,(k-i,i)}(\xbf_1,\xbf_2)$ for all but two indices $i=I, I+1$. In the end, we obtain linear forms $\Psi_m(\hbf_1,...,\hbf_{k-1})$, defined in \eqref{eqn:def_psi},  with respect to a single, symmetric, multilinear form $\Phi_k$, instead of a system  of $k+1$ forms (e.g. as in \cite[Lemma 2.1]{Bir62}).

\begin{prop}\label{lem:weyl2}
  Let $\underline{s}=(s_1,...,s_{k-1})\in \{1,2\}^{k-1}$ where $s_i=1$ for $1\leq i \leq k-(I+1)$ and $s_i=2$ for $k-I\leq i \leq k-1$.
  Then
    \begin{align*}
          |S(\alpha)|^{2^{k-1}}\ll_n  (L_1^n L_2^n)^{2^{k-1} -1}
        (\prod_{i=1}^{k-1} L_{s_i}^n )^{-1} L_{2}^n\sum_{\substack{\hbf_1,...,\hbf_{k-1} \\ \hbf_i\in [-\Lbf_{s_i},\Lbf_{s_i}]}} \prod_{m=1}^n \min \{2L_{1},\|\alpha_{I} C(k) \Psi_m(\hbf_1,...,\hbf_{k-1}) \|^{-1}\},       
    \end{align*}
where $C(k)$ is a nonzero constant depending on $k$ and $\Psi_m$ is defined by
    \begin{align}\label{eqn:def_psi}
         \tilde{\Phi}_{k,(1,1,...,1)}(\xbf,\hbf_1,...,\hbf_{k-1})=\sum_{m=1}^n x_m \Psi_m(\hbf_1,...,\hbf_{k-1}),
    \end{align}
where recall the definition of $\tilde{\Phi}_{k,(1,1,...,1)}$ from \eqref{def:tilde_phi}.
\end{prop}
    \begin{proof}
    As before, let $\Delta^{\kappa}_{ \underline{s},\underline{\hbf}} = \Delta_{s_{\kappa},\hbf_{\kappa}}\cdots \Delta_{s_1,\hbf_1}$.
    Apply Lemma \ref{lem:weyl_differencing} with $\kappa = k-1$ and $\underline{s}$ as in the hypothesis, so that
    \begin{align*}
        |S(\alpha)|^{2^{k-1}}\ll_n
        (L_1^n L_2^n)^{2^{k-1} -1}
        (\prod_{i=1}^{k-1} L_{s_i}^n)^{-1} \sum_{\substack{\hbf_1,...,\hbf_{k-1} \\ \hbf_i\in [-\Lbf_{s_i},\Lbf_{s_i}]}}\sum_{\substack{\xbf_1,\xbf_2 \\ \xbf_i\in B^{k-1}_{i}}}e(\Delta^{k-1}_{ \underline{s},\underline{\hbf}}\Tilde{F}(\xbf_1,\xbf_2;\alpha)),
    \end{align*}
where, by definition \eqref{def:tildeF},
    \begin{align*}
      \Delta^{k-1}_{ \underline{s},\underline{\hbf}} \Tilde{F}(\xbf_1,\xbf_2;\alpha) = \sum_{j=0}^k  \sum_{i=0}^j \alpha_i  \binom{j}{i} \Delta^{k-1}_{ \underline{s},\underline{\hbf}}\Tilde{\Phi}_{j,(j-i,i)}(\xbf_1,\xbf_2).
    \end{align*}
    In the notation of Lemma \ref{lem:delta_explicit}, $\sigma^1_{k-1,\underline{s}} = k-(I+1)$ and $\sigma^2_{k-1,\underline{s}} = I$, and for fixed $j,i$, we have $\xi_0^1 = j-i$ and $\xi_0^2=i$.
For $0\leq j\leq k-2$ and $0\leq i\leq j$, either $j-i<k-(I+1)$ or $i<I$, or both, and hence by Lemma \ref{lem:delta_explicit}, $\Delta_{ \underline{s},\underline{\hbf}}^{k-1}\tilde{\Phi}_{j,(j-i,i)}$ vanishes.  For $j=k-1$ and $0\leq i \leq j$, $\Delta^{k-1}_{ \underline{s},\underline{\hbf}}\tilde{\Phi}_{j,(j-i,i)}$ vanishes unless $i=I$, in which case it has bihomogeneous degree $(0,0)$. Finally when $j=k$, $\Delta^{k-1}_{ \underline{s},\underline{\hbf}}\tilde{\Phi}_{j,(j-i,i)}$ vanishes unless $i\geq I$ and $k-i \geq k-(I+1)$. This occurs when $i=I$ or $i=I+1$, in which case $\Delta^{k-1}_{ \underline{s},\underline{\hbf}}\tilde{\Phi}_{k,(k-i,i)}$ has bihomogeneous degree $(1,0)$ and $(0,1)$, respectively. More precisely, by \eqref{eqn:delta_explicit2},
    \begin{multline*}
       \Delta^{k-1}_{ \underline{s},\underline{\hbf}}\tilde{\Phi}_{k,(k-i,i)}(\xbf_1,\xbf_2)
                = \sum_{\iota_{k-1}=0}^{\xi_{k-2}^{s_{k-1}}-1}\cdots \sum_{\iota_1=0}^{\xi_0^{s_1}-1}\binom{\xi_{k-2}^{s_{k-1}}}{\iota_{k-1}}\cdots \binom{\xi_0^{s_1}}{\iota_1} \\
          \cdot  \tilde{\Phi}_{k,(\xi_{k-1}^{1},\xi_{k-1}^{2},\eta_{k-1}^{1},...,\eta_{k-1}^{k-1})}(\xbf_1,\xbf_2,\hbf_1,...,\hbf_{k-1}),
    \end{multline*}
    where, for $i=I$ and $I+1$, the respective linear term occurs when $\iota_\kappa=\xi^{s_\kappa}_{\kappa-1}-1$ for all $1\leq \kappa \leq k-1$. By Lemma \ref{lem:delta_explicit}, we also have $\eta_{k-1}^{\kappa}\geq 1$. Since $\tilde{\Phi}_{k,(\xi_{k-1}^{1},\xi_{k-1}^{2},\eta_{k-1}^{1},...,\eta_{k-1}^{k-1})}$ is a degree $k$ multilinear form, we deduce that $\eta_{k-1}^{1}+\cdots + \eta_{k-1}^{k-1}=k-1$, and hence $\eta_{k-1}^{\kappa}=1$ for all $1\leq \kappa \leq k-1$. Thus when $\iota_\kappa=\xi^{s_\kappa}_{\kappa-1}-1$ for all $1\leq \kappa \leq k-1$, we can rewrite the linear term as
    \begin{align*}
         \tilde{\Phi}_{k,(\xi_{k-1}^{1},\xi_{k-1}^{2},1,...,1)} =
        \begin{cases}
         \sum_{m=1}^n x_{1,m}\Psi_m(\hbf_1,...,\hbf_{k-1}),    &i=I\\
         \sum_{m=1}^n x_{2,m}\Psi_m(\hbf_1,...,\hbf_{k-1}),    &i=I+1,
        \end{cases}
    \end{align*}
where $\Psi_m$ is multilinear in the $\hbf_i$'s. 
Then, combining these observations,
    \begin{align*}
        \Delta^{k-1}_{\underline{s}, \underline{\hbf}} \Tilde{F}(\xbf_1,\xbf_2;\alpha) 
        =\sum_{i=I}^{I+1} \alpha_{i}C'(i,k)\binom{k}{i}\sum_{m=1}^n x_{i-(I-1),m}\Psi_m(\hbf_1,...,\hbf_{k-1}) + R,
    \end{align*}
where $C'(i,k)$ is the product of binomial coefficients when $\iota_\kappa=\xi^{s_\kappa}_{\kappa-1}-1$ for all $1\leq \kappa \leq k-1$
and $R=R(k,\underline{s},\underline{\hbf})$ encapsulates all other terms.
Then, upon bounding the sum over $\xbf_2$ trivially, we get
    \begin{align*}
| \sum_{\substack{\xbf_1,\xbf_2 \\ \xbf_i\in B^{k-1}_{i}}}e( \Delta^{k-1}_{ \underline{s}, \underline{\hbf}}\Tilde{F}(\xbf_1,\xbf_2;\alpha)) |
      \leq  
     |B^{k-1}_{2} |\cdot |\sum_{\xbf_{1}\in B^{k-1}_{1}}e( \alpha_{I}C(k)\sum_{m=1}^n x_{1,m}\Psi_m(\hbf_1,...,\hbf_{k-1}))|,
    \end{align*}
    where applying the standard bound on one-dimensional exponential sums with linear phase completes the proof.
    \end{proof}

This completes the initial step of bounding $S(\alpha)$ from above by an expression depending on the multilinear forms $\Psi_m(\hbf_1,...,\hbf_{k-1})$, $1\leq m \leq n$. 

\subsection{What happens when the exponential sum is large}\label{sec:expsum_large}
In this section, we continue the standard next step of showing that if $|S(\alpha)|$ is large, then $\|\alpha_{I} C(k)\Psi_m(\hbf_1,...,\hbf_{k-1}) \|$ must be small for all $1\leq m \leq n$, for many tuples $(\hbf_1,...,\hbf_{k-1})$.
For $M_1,...,M_{k-1},M>0$, define
    \begin{multline*}
        N(M_1,...,M_{k-1}; M^{-1}) 
        = |\{(\hbf_1,...,\hbf_{k-1})\in \prod_{i=1}^{k-1}[-M_i,M_i]^n: \\ \|\alpha_I C(k)\Psi_m(\hbf_1,...,\hbf_{k-1})\|<M^{-1}, 1 \leq m \leq n\}|.
    \end{multline*}

\begin{lemma}\label{lem:N1}
 Let $(s_1,...,s_{k-1})\in \{1,2\}^{k-1}$ where $s_i=1$ for $1\leq i \leq k-(I+1)$ and $s_i=2$ for $k-I\leq i \leq k-1$.  Suppose $S(\alpha)\gg_n L_2^n L_1^{n-\delta}$ for some $\delta>0$. Then 
    \begin{align*}
        N(4L_{s_1},...,4L_{s_{k-1}};L_1^{-1}) \gg_n  (\prod_{i=1}^{k-1}L_{s_i}^n) L_1^{-2^{k-1}\delta} (\log L_1)^{-n}.
    \end{align*}
\end{lemma}

\begin{proof}
This follows the same steps as in \cite[Lemma 13.2]{Dav05}.
By hypothesis and Proposition \ref{lem:weyl2},
        \begin{align}\label{eqn:counting_h}
            \sum_{\substack{\hbf_1,...,\hbf_{k-1} \\ \hbf_i \in [-\Lbf_{s_i}, \Lbf_{s_i}]}} \prod_{m=1}^n \min \{2L_1,\|\alpha_I C(k)\Psi_m(\hbf_1,...,\hbf_{k-1}) \|^{-1}\}
            \gg_n L_1^{n-2^{k-1}\delta} \prod_{i=1}^{k-1}L_{s_i}^n. 
        \end{align}
        We give an upper bound on the left-hand side in terms of $N(4L_{s_1},...,4L_{s_{k-1}};L_1^{-1})$. Fix $\hbf_1,...,\hbf_{k-2}$ with $\hbf_i\in [-\Lbf_{s_i},\Lbf_{s_i}]$ for $1\leq i \leq k-2$. Decompose the last box of summation as follows:
            \begin{align*}
           [-\Lbf_{s_{k-1}},\Lbf_{s_{k-1}}] = \bigcup_{r_1=0}^{L_1-1} \cdots \bigcup_{r_n=0}^{L_1-1} S_{r_1,...,r_n}, 
            \end{align*}
        where $ S_{r_1,...,r_n} = S_{r_1,...,r_n}(\hbf_1,...,\hbf_{k-2})$ is the set of $\hbf_{k-1} \in  [-\Lbf_{s_{k-1}},\Lbf_{s_{k-1}}]$ such that
         \begin{align*}
        r_m L_1^{-1}< \mathrm{frac}(\alpha_{I} C(k)\Psi_m(\hbf_1,...,\hbf_{k-1}))\leq (r_m+1)L_1^{-1}, \quad 1\leq m \leq n;
        \end{align*}
        here, for a real number $\alpha$, define $\mathrm{frac}(\alpha)$ to be its fractional part. We compare this to the following set.
    Define $S' = S'(\hbf_1,...,\hbf_{k-2})$ to be the set of $\hbf_{k-1}\in [-4L_{s_{k-1}},4L_{s_{k-1}}]^n$ such that
        \begin{align*}
           \|\alpha_{I} C(k) \Psi_m(\hbf_1,...,\hbf_{k-1})\|<L_1^{-1}, \quad 1\leq m \leq n
        \end{align*}
    so that 
        \begin{align}\label{eqn:N_sumA}
            N(4L_{s_1},...,4L_{s_{k-1}};L_1^{-1}) \geq \sum_{\substack{\hbf_1,...,\hbf_{k-2} \\ \hbf_i\in [-\Lbf_{s_i},\Lbf_{s_i}]}} |S'(\hbf_1,...,\hbf_{k-2})|.
        \end{align}
    We claim that if $0\leq r_m\leq L_1-1$ for all $1\leq m \leq n$, then 
        \begin{align}\label{eqn:A2A1}
            |S_{r_1,...,r_n}(\hbf_1,...,\hbf_{k-2})|\leq |S'(\hbf_1,...,\hbf_{k-2})|.
        \end{align}
    Indeed, fix an element $\hbf_{k-1}\in S_{r_1,...,r_n}$, and let $\hbf'_{k-1}\in S_{r_1,...,r_n}$ be another element. It suffices to show that $\hbf_{k-1} - \hbf'_{k-1}\in S'$. First, we check that $|h_{k-1,j} - h_{k-1,j}'|\leq 4L_{s_{k-1}}$. Next, note that for real numbers $\alpha, \beta$, $|\mathrm{frac}(\alpha - \beta)| \leq \max \{|\mathrm{frac}(\alpha) - \mathrm{frac} (\beta)|,1-|\mathrm{frac}(\alpha) - \mathrm{frac} (\beta)|\}$ and so $\|\alpha - \beta\| \leq \|\mathrm{frac}(\alpha) - \mathrm{frac} (\beta)\|$. Hence by the definition of $S_{r_1,...,r_n}$,
        \begin{align*}
        \|\alpha_{I} C(k)\Psi_m(\hbf_1,...,\hbf_{k-2,}\hbf_{k-1}-\hbf'_{k-1})\| \leq L_1^{-1}, \qquad 1 \leq m \leq n
        \end{align*}
    which proves the claim.

    We return to the expression \eqref{eqn:counting_h}; denote the left-hand side as $\Sigma$. Then
        \begin{align*}
        |\Sigma|&\leq  \sum_{\substack{\hbf_1,...,\hbf_{k-2} \\ \hbf_i\in [-\Lbf_{s_i},\Lbf_{s_i}]}} \sum_{\substack{r_1,...,r_n \\ 0\leq r_i \leq L_1-1}}\sum_{\hbf_{k-1}\in S_{r_1,...,r_n}}\prod_{m=1}^n \min \{2L_1,\|\alpha_{I} C(k) \Psi_m(\hbf_1,...,\hbf_{k-1}) \|^{-1}\}\\
        &\ll_n 
       \sum_{\substack{\hbf_1,...,\hbf_{k-2} \\ \hbf_i\in [-\Lbf_{s_i},\Lbf_{s_i}]}} \sum_{\substack{r_1,...,r_n\\ 0\leq r_i\leq L_1-1}} |S_{r_1,...,r_n}| \prod_{m=1}^n \min\left\{ L_1,\frac{L_1}{r_m},\frac{L_1}{L_1-(r_m+1)}\right\}.
        \end{align*}
    By \eqref{eqn:A2A1} and \eqref{eqn:N_sumA}, 
       $  |\Sigma| \leq N(4L_{s_1},...,4L_{s_{k-1}};L_1^{-1})  (L_1\log L_1)^n$,
    and hence together with \eqref{eqn:counting_h},
        \begin{align*}
          N(4L_{s_1},...,4L_{s_{k-1}};L_1^{-1})
           \gg_n L_1^{-2^{k-1}\delta} (\prod_{i=1}^{k-1}L_{s_i}^n)  (\log L_1)^{-n},
        \end{align*}
    which concludes the proof.
\end{proof}

\begin{prop}\label{prop:weyl3} 
Let $0<L_2\leq L_1$ and $0<\theta \leq \log L_2 / \log L_1$.
  Suppose $S(\alpha)\gg  L_2^n L_1^{n-\delta}$ for some $\delta>0$. Let $(s_1,...,s_{k-1})\in \{1,2\}^{k-1}$ where $s_i=1$ for $1\leq i \leq k-(I+1)$ and $s_i=2$ for $k-I\leq i \leq k-1$. Define $\lambda_\kappa = L_1/L_{s_\kappa}$ for each $1\leq \kappa \leq k-1$. Then
    \begin{align*}
        N(4L_1^{\theta},...,4L_1^{\theta};L_1^{-k+(k-1)\theta} \lambda_{s_{1}}\cdots \lambda_{s_{k-1}})\gg_n L_1^{\theta(k-1)n -2^{k-1}\delta} 
        (\log L_1)^{-n}.
    \end{align*}
\end{prop}

We record the following result of Davenport.
\begin{lemma}[{\cite[Lemma 3.3]{Dav59}}]\label{lem:LA}
     Let $P_1,...,P_n\in \bR[X_1,...,X_n]$ be linear forms given by $P_i=a_{i,1}X_1+\cdots + a_{i,n}X_n$ such that $a_{i,j}=a_{j,i}$. Let $A>1$ be real. Define
        \begin{align*}
            U(Z) = |\{(u_1,...,u_n)\in \mathbb{Z}^{n}: |u_i|<AZ \text{ and } \|P_i(u_1,...,u_n)\|<A^{-1}Z, \ 1 \leq i \leq n\}|.
        \end{align*}
    Then for $0<Z'\leq Z\leq 1$, we have $U(Z)/U(Z') \ll (Z/Z')^n$.
\end{lemma}

\begin{proof}[Proof of Proposition \ref{prop:weyl3}]
 For each $1\leq \kappa \leq k-1$ and $\hbf_1,...,\hbf_{\kappa -1}, \hbf_{\kappa +1},...,\hbf_{k-1}$, let
        \begin{align*}
            U_\kappa(A,Z) = |\{\hbf_{\kappa}\in [-AZ,AZ]^n: \|\alpha_I C(k) \Psi_m(\hbf_1,...,\hbf_{k-1})\|<A^{-1}Z, \ 1 \leq m \leq n\}|,
        \end{align*}
    where we view $\Psi_m(\hbf_1,...,\hbf_{k-1})$ as a linear form in $\hbf_{\kappa}$. Note that the condition $a_{ij}=a_{ji}$ in Lemma \ref{lem:LA} is satisfied by definition of $\Psi_m$ from \eqref{eqn:def_psi} and symmetricity of $\tilde{\Phi}_{k,(1,1,...,1)}$.
    We will apply Lemma \ref{lem:LA} $k-1$ times, to each $\kappa$, with
        \begin{align*}
        A_\kappa &= 2(L_1^{(\kappa+1) - (\kappa -1)\theta} \lambda_{s_{1}}^{-1}\cdots \lambda_{s_{\kappa}}^{-1})^{1/2},\\
            Z_\kappa &= 2(L_1^{-(\kappa-1) + (\kappa -1)\theta} \lambda_{s_{1}}\cdots \lambda_{s_{\kappa -1}}\lambda_{s_{\kappa}}^{-1})^{1/2},\\
            Z_\kappa' &= 2(L_1^{-(\kappa+1) +(\kappa+1)\theta}\lambda_{s_{1}}\cdots \lambda_{s_{\kappa}})^{1/2}.
        \end{align*}
    We check that $Z'_\kappa\leq Z_\kappa \leq 1$ holds since $L_1^\theta \leq L_i\leq L_1$ for $i=1,2$, which is true by assumption. With the above choices, we have $ A_\kappa Z_\kappa=4L_{s_{\kappa}}$, $A_\kappa^{-1} Z_\kappa=L_1^{-\kappa+(\kappa-1)\theta} \lambda_{s_{1}}\cdots \lambda_{s_{\kappa-1}}$, and $  A_\kappa Z'_\kappa=4L_1^\theta$, $A_\kappa^{-1} Z'_\kappa=L_1^{-(\kappa+1)+\kappa\theta} \lambda_{s_{1}}\cdots \lambda_{s_{\kappa}}$.
    Then, applying Lemma \ref{lem:LA} gives 
        \begin{align}\label{eqn:apply_lem_LA}
            U_\kappa(A_\kappa, Z'_\kappa) \gg U_\kappa(A_\kappa, Z_\kappa) (\frac{Z'_\kappa}{Z_\kappa})^n \gg U_\kappa(A_\kappa, Z_\kappa) (L_1^{-1+\theta} \lambda_{s_{\kappa}})^n.
        \end{align}
    We claim that for $2\leq \kappa\leq k-1$,
        \begin{multline}\label{eqn:ineq_N}
            N(4L_1^{\theta},...,4L_1^{\theta},4L_{s_{\kappa+1}},...,4L_{s_{k-1}};A^{-1}_\kappa Z'_\kappa) \\
            \gg (L_1^{-1+\theta}\lambda_{s_{\kappa}})^n N(4L_1^{\theta},...,4L_1^{\theta},4L_{s_\kappa},...,4L_{s_{k-1}};A^{-1}_{\kappa-1} Z'_{\kappa-1}).
        \end{multline}
Indeed, by \eqref{eqn:apply_lem_LA}, the left-hand side is
        \begin{align*}
       \sum_{\substack{\hbf_i\in [-4L_1^{\theta},4L_1^{\theta}]^n\\ 1 \leq i \leq \kappa-1}}   \sum_{\substack{\hbf_i\in [-4L_{s_i},4L_{s_i}]^n\\ \kappa+1 \leq i \leq k-1}}     U_\kappa (A_\kappa,Z'_\kappa) 
        \gg  (L_1^{-1+\theta} \lambda_{s_{\kappa}})^n  \sum_{\substack{\hbf_i\in [-4L_1^{\theta},4L_1^{\theta}]^n\\ 1\leq i \leq \kappa-1}} \sum_{\substack{\hbf_i\in [-4L_{s_i},4L_{s_i}]^n\\ \kappa+1 \leq i \leq k-1}}       U_\kappa (A_\kappa, Z_\kappa).
        \end{align*}
    Since $A^{-1}_\kappa Z_\kappa = A^{-1}_{\kappa-1} Z'_{\kappa-1}$, the double sum on the right-hand side is precisely 
        \begin{align*}
            N(4L_1^{\theta},...,4L_1^{\theta},4L_{s_\kappa},...,4L_{s_{k-1}};A^{-1}_{\kappa-1} Z'_{\kappa-1}),
        \end{align*}
    which proves the claim. 
    Iterating \eqref{eqn:ineq_N} $k-2$ times and then \eqref{eqn:apply_lem_LA} once, we obtain
        \begin{align*}
            N(4L_1^{\theta},...,4L_1^{\theta};A^{-1}_{k-1} Z'_{k-1}) 
           & \gg L_1^{(-1+\theta)(k-2)n}(\lambda_{s_{2}}\cdots \lambda_{s_{k-1}})^n N(4L_1^\theta,4L_{s_2},...,4L_{s_{k-1}};A^{-1}_1 Z'_{1})\\
          & \gg L_1^{(-1+\theta)(k-1)n}(\lambda_{s_{1}}\cdots \lambda_{s_{k-1}})^n N(4L_{s_1},...,4L_{s_{k-1}};L_1^{-1}).
        \end{align*}
    Finally, apply Lemma \ref{lem:N1} to the right-hand side to get
        \begin{multline*}
            N(4L_1^{\theta},...,4L_1^{\theta};L_1^{-k+(k-1)\theta} \lambda_{s_{1}}\cdots \lambda_{s_{k-1}})\\
            \gg_n L_1^{(-1+\theta)(k-1)n}(\lambda_{s_{1}}\cdots \lambda_{s_{k-1}})^n
        L_1^{-2^{k-1}\delta} (\prod_{i=1}^{k-1} L_{s_i})^n(\log L_1)^{-n},
        \end{multline*}
    which completes the proof.
\end{proof}

\subsection{Proof of Theorem \ref{thm:expsum_1}}\label{sec:pf_22}
We record a version of the following result of Schmidt, which bounds the number of points in the intersection of a variety over $\bF_p$ with a ``box'' in $\bF_p$. This is essentially a finite field analogue of \cite[Lemma 3.3]{Bir62} and a generalization of the Schwartz-Zippel lemma for varieties over finite fields. We refer readers to e.g. \cite[Theorem 3.1]{Bro09} for a proof in the integers case which is analogous to the finite field case.
\begin{lemma}[{\cite[Equation (6.5)]{Sch84_expsums}}]\label{lem:schmidt} Let $r,l\geq 1$ and $d\geq 0$ be integers.
Let $D(Z)\subseteq (\bF_p)^r$ be a subset whose projection onto any coordinate axis has cardinality at most $Z$.
    Let $V\subseteq (\overline{\bF}_p)^r$ be an algebraic set defined by the zero locus of polynomials $g_1,...,g_l$, where $\deg g_i\leq l$ for all $i$, and suppose $V$ has dimension $d$. Then $|V\cap D(Z)|\ll_{r,l} Z^d$.
\end{lemma}

We now deduce the following consequence of Proposition \ref{prop:weyl3}.

\begin{prop}\label{prop:tripartite} Let $n,k$ be positive integers with $n\geq k-1$.
Let $p$ be a prime such that $C(k)<p$, where $C(k)$ is from Proposition \ref{lem:weyl2}. Let $\alpha = (\alpha_0,...,\alpha_k)$ with $\alpha_i=a_i/p$, where $0\leq a_i \leq p-1$, and suppose $a_I\not\equiv 0 \modd p$ for some $0\leq I \leq k$.
Let $0<L_2\leq L_1 \leq p$ with $L_1>p^{1/k} \lambda_2^{1-1/k}$, and let $0<\theta < \log (L_1^k p^{-1}\lam_2^{1-k}) / \log (L_1^{k-1})$. 
Then, either
        \begin{enumerate}
            \item $S(\alpha)\ll_n L_2^n L_1^{n-\delta}$ for some $\delta>0$, or
            \item for all $\varepsilon>0$, $\dim_{\overline{\bF}_p} \sing F_k \geq n - 2^{k-1}\delta/\theta - \varepsilon$.
        \end{enumerate}
\end{prop}
    \begin{proof}
        Suppose the first bound does not hold.
        To apply Proposition \ref{prop:weyl3}, we check that
        \begin{align*}
          0<   \log (L_1^k p^{-1}\lam_2^{1-k}) / \log (L_1^{k-1}) \leq \log L_2 / \log L_1
        \end{align*}
        as long as $p^{1/k}\lam_2^{1-1/k}<L_1\leq p$, which we have assumed. 
        Then by Proposition \ref{prop:weyl3}, for $(s_1,...,s_{k-1})$ with $s_i=1$ for $1\leq i \leq k-(I+1)$ and $s_i=2$ for $k-I\leq i \leq k-1$, \begin{align*}
            \|\alpha_{I} C(k) \Psi_m(\hbf_1,...,\hbf_{k-1})\| \leq L_1^{-k+(k-1)\theta} \lambda_{s_{1}}\cdots \lambda_{s_{k-1}}, \qquad 1\leq m \leq n
        \end{align*}
         for at least $L_1^{\theta(k-1)n - 2^{k-1}\delta} (\log L_1)^{-n}$ many choices of $(\hbf_1,...,\hbf_{k-1})\in [-4L_1^{\theta},4L_1^{\theta}]^{(k-1)n}$. Note that $\lambda_{s_{1}}\cdots \lambda_{s_{k-1}} \leq \lam_2^{k-1}$, and by our choice of $\theta$, we have $ pL_1^{-k+(k-1)\theta} \lam_2^{k-1}<1.$
         Then for all $1\leq m \leq n$, $\|\alpha_{I} C(k) \Psi_m(\hbf_1,...,\hbf_{k-1})\| < p^{-1}$. But since $\alpha_I=a_I/p$, and $C(k)$ and $\Psi_m(\hbf_1,...,\hbf_{k-1})$ are integers, we deduce that $a_I C(k) \Psi_m(\hbf_1,...,\hbf_{k-1})\equiv 0 \modd p$. Since $a_I\not\equiv 0 \modd p$ and $C(k)<p$, we further deduce that 
            \begin{align*}
                \Psi_m(\hbf_1,...,\hbf_{k-1})\equiv 0 \modd p, \qquad 1 \leq m \leq n
            \end{align*}
               for at least (asymptotically, with constant depending on $n,\varepsilon$) $L_1^{\theta(k-1)n - 2^{k-1}\delta - \varepsilon}$ many choices of $(\hbf_1,...,\hbf_{k-1})\in [-4L_1^{\theta},4L_1^{\theta}]^{(k-1)n}$.

    Let $\mathscr{S}$ denote the locus of tuples $(\hbf_1,...,\hbf_{k-1})\modd p ^{(k-1)n}$ such that $\Psi_m(\hbf_1,...,\hbf_{k-1})\equiv 0\modd p$ for all $1\leq m \leq n$. By Lemma \ref{lem:schmidt} with $r=(k-1)n, \ l=n, \ V=\mathscr{S}, \ Z=8L_1^{\theta}$, we have $\dim \mathscr{S} \geq (k-1)n - 2^{k-1}\delta/\theta - \varepsilon$. Let $\mathscr{D}$ denote the diagonal $\{(\hbf_1,...,\hbf_{k-1}) \modd p^{(k-1)n}: \hbf_1=\cdots = \hbf_{k-1}\}$, which is defined by $(k-2)n$ linear equations mod $p$. Then $\dim (\mathscr{D}\cap \mathscr{S}) \geq \dim \mathscr{S} - (k-2)n$. More precisely,
        \begin{align*}
            \mathscr{D}\cap \mathscr{S} = \{(\hbf,...,\hbf) \modd p ^{(k-1)n}: \Psi_m(\hbf,...,\hbf)\equiv 0 \modd p, \ 1 \leq m \leq n\}.
        \end{align*}
    By property \eqref{eqn:property_partial_binom} and the definition of $\Psi_m$ in \eqref{eqn:def_psi}, we have the identity
        \begin{align*}
            \Psi_m(\hbf,...,\hbf) = k^{-1} \frac{\partial F_k}{\partial X_m}(\hbf).
        \end{align*}
        Since $C(k)<p$, $\mathscr{D}\cap \mathscr{S}$ projects onto $  \{\hbf \modd p ^n: (\partial F_k / \partial X_m )(\hbf)\equiv 0\modd p, \ 1\leq m \leq n\}$,
    which is precisely the singular locus of $F_k=0$ over $\bF_p$. We conclude with the condition that $\dim \sing F_k \geq n - 2^{k-1}\delta/\theta - \varepsilon$.
\end{proof}

Finally, we apply Proposition \ref{prop:tripartite} to
    \begin{align*}
        S(\abf;F)=\sum_{(\xbf,\ybf)\in (\Nbf-\Hbf,\Nbf+\Hbf]\times (0,K]^n} e_p(\sum_{i=0}^k a_i f_i(\xbf,\ybf)),
    \end{align*}
from \eqref{eqn:S1}, with $L_1=2H$ and $L_2=K$. Let $\varepsilon>0$ and
set
    \begin{align*}
        \delta = (\frac{n-\dim \sing F_k}{2^{k-1}})\theta - \varepsilon
    \end{align*}
so that $n-2^{k-1}\delta/\theta - \varepsilon > \dim \sing F_k$. This eliminates the possibility of the second case, and so we conclude a nontrivial bound on $S(\abf;F)$ of $K^n H^{n-\delta}$. This completes the proof of Theorem \ref{thm:expsum_1}.

\section{Proof of Theorem \ref{thm:expsum_2}}
\label{sec:expsum_2}

We record the following consequence of the resolution of the main conjecture of Vinogradov's mean value theorem, as recorded in \cite[Theorem 5]{Bou17} (see also \cite[\S 4.1, Theorem 4]{Mon94}).
\begin{thm}\label{thm:weyl} Fix an integer $d\geq 3$, and let $f(x) = \alpha_d x^d+\cdots + \alpha_1 x$. Suppose for some $2\leq j \leq d$ that $|\alpha_j-b/q|<q^{-2}$ for some $(b,q)=1$. Then for all $\varepsilon>0$,
        \begin{align*}
       \sum_{1\leq x \leq H} e(f(x))  \ll_\varepsilon H^{1+\varepsilon}(q^{-1}+H^{-1}+qH^{-j})^{\sigma(d)},
        \end{align*}
    where $\sigma(d) = 1/d(d-1)$.
\end{thm}

Recall from \eqref{eqn:S1} that we want to bound
    \begin{align}\label{eqn:bourgain_bound}
\sum_{\substack{(a_0,...,a_{k})\not\equiv \mathbf{0}\\ \modd p^{k+1}}}|\sum_{(\xbf,\ybf)\in (\Nbf-\Hbf,\Nbf+\Hbf]\times (0,K]^{n}} e_p(\sum_{i=0}^{k} a_i f_i(\xbf,\ybf))|^2,
    \end{align}
   where $ f_i(\xbf,\ybf) = \sum_{ |\beta|=i}(\partial^\beta F)(\xbf) \ybf^\beta / \beta !$ and $F = G_1(X_1)+\cdots + G_n(X_n)$. For each $1\leq j\leq n$, write $G_j(X_j) = c_{j,k}X_j^{k}+\cdots +c_{j,1}X_j + c_{j,0}$, where $c_{j,k}\not\equiv 0 \modd p$, so that the $i$-th derivative is $G_j^{(i)}(X_j) = \sum_{r=i}^{k} i!\binom{r}{i} c_{j,r}X_j^{r-i}$.
Since $(\partial^2 F/\partial X_i \partial X_j)$ is identically zero for all $i\neq j$, we have
 \begin{align*}
   f_i(\xbf,\ybf) = \sum_{ j=1}^{n}  \frac{ G_j^{(i)}(x_j)}{i!} y_j^i
   =\sum_{ j=1}^{n}  (\sum_{r=i}^{k} \binom{r}{i} c_{j,r}x_j^{r-i}) y_j^i,   
    \end{align*}
and so 
\eqref{eqn:bourgain_bound} is
    \begin{align}\label{eqn:bourgain_bound2}
     \sum_{\substack{(a_0,...,a_{k})\not\equiv \mathbf{0}\\ \modd p^{k+1}}}
     \prod_{j=1}^{n} | \sum_{y_j\in  (0,K]} \sum_{x_j\in (N_j-H_j,N_j+H_j]}  e_p(\sum_{i=0}^{k} a_i  \sum_{r=i}^{k} \binom{r}{i} c_{j,r}x_j^{r-i} y_j^i)|^2.
    \end{align}
For each nonzero $\abf=(a_0,...,a_{k})$ and $1\leq j\leq n$, let $S_j(\abf;F)$ denote the inner double sum.
We shall estimate $S_j(\abf;F)$ by fixing one of the variables $x_j,y_j$ and applying Theorem \ref{thm:weyl} to the remaining variable. Before proceeding, we note two subtleties. First, as $(a_0,...,a_k)$ varies, in particular, when certain coordinates are zero mod $p$, the degree and leading coefficient of the associated polynomial change, and so we must be careful to ensure that we apply Theorem \ref{thm:weyl} correctly. Second, the upper bound in Theorem \ref{thm:weyl} is dominated by $H^{1-\sigma(d)+\varepsilon}$ when $H\gg q^{1/(j-1)}$. This is a reasonable range when $j$ is large. With these considerations, we decompose the set $(\bZ/p\bZ)^{k+1}\setminus \{\mathbf{0}\}$ as follows.
    For $0\leq \kappa \leq k$, define
    \begin{align*}
        A_\kappa = \{(a_0,...,a_{k})\modd p ^{k+1}: a_0,a_1,...,a_{\kappa-1}\equiv 0, a_\kappa\not\equiv 0 \modd p\},
    \end{align*}
    so that $ (\bZ/p\bZ)^{k+1}\setminus \{\mathbf{0}\} = \bigcup_{0\leq \kappa \leq k} A_\kappa$.
For $\lfloor k/2 \rfloor  < \kappa \leq k$, define
    \begin{align*}
        A_\kappa' = \{(a_0,...,a_{k})\modd p ^{k+1}: a_{k},a_{{k}-1},...,a_{\kappa+1}\equiv 0, a_\kappa\not\equiv 0 \modd p\},
    \end{align*}
    and notice that
    \begin{align*}
        \bigcup_{\lfloor k/2 \rfloor <\kappa\leq k}A_\kappa \subseteq  \bigcup_{\lfloor k/2 \rfloor <\kappa\leq k}A_\kappa'.
    \end{align*}
Indeed, let $\abf$ be an element of the left-hand side. Then $a_\kappa\not\equiv 0 \modd p$ for some $\lfloor k/2 \rfloor <\kappa\leq k$.
The claim follows by noting that the complement of the right-hand side is the set of $\abf$ such that $a_{k},a_{k-1},...,a_{\lfloor k/2\rfloor +1}\equiv 0 \modd p$.
Then by this decomposition, \eqref{eqn:bourgain_bound2}, and positivity, \eqref{eqn:bourgain_bound} is at most
    \begin{align*}
        \sum_{0\leq \kappa \leq \lfloor k/2 \rfloor } \sum_{\abf\in A_\kappa}\prod_{j=1}^{n}|S_j(\abf;F)|^2 
        + \sum_{\lfloor k/2 \rfloor < \kappa \leq k} \sum_{\abf\in A_\kappa'}\prod_{j=1}^{n}|S_j(\abf;F)|^2.
    \end{align*}
Now we bound the two terms separately.

   Fix $0\leq \kappa \leq \lfloor k/2 \rfloor$, and let $\abf\in A_\kappa$, so that $a_0,a_1,...,a_{\kappa-1}\equiv 0$ and $a_\kappa\not\equiv 0 \modd p$. Then
    \begin{align*}
    S_j(\abf;F)= \sum_{y_j\in (0,K]}   \sum_{x_j\in (0,2H_j]} e_p(\sum_{i=\kappa}^{k} a_i  \sum_{r=i}^{k} \binom{r}{i} c_{j,r}(x_j+N_j-H_j)^{r-i} y_j^i).
    \end{align*}
    Fix $y_j\in (0,K]$. Apply Theorem \ref{thm:weyl} to the sum over $x_j$, where the argument is a single-variable polynomial in $x_j$ of degree $k-\kappa$. Note that $k-\kappa\geq k-\lfloor k/2 \rfloor \geq 3$ as long as $k\geq 5$. We check that the leading coefficient $a_\kappa \binom{k}{\kappa}c_{j,k}y_j^\kappa$ is nonzero mod $p$, since $a_\kappa, \binom{k}{\kappa}, c_{j,k}, y_j$ are all nonzero mod $p$, and hence
 \begin{align*}
            S_j(\abf;F) \ll_\varepsilon  KH_j^{1+\varepsilon}(p^{-1}+H_j^{-1}+pH_j^{-(k-\kappa)})^{\sigma(k-\kappa)}.
        \end{align*}
   Note that $H_j^{-1}\gg p^{-1}$, since we assume $H_j\ll p$. On the other hand, $H_j^{-1}\geq pH_j^{-(k-\kappa)}$ if and only if
        \begin{align}\label{eqn:A_regime}
            H_j\geq p^{1/(k-\kappa-1)}.
        \end{align}
   In this regime, we have the bound
        \begin{align}\label{eqn:A_bound}
            S_j(\abf;F) \ll_{\varepsilon} KH_j^{1-\sigma(k-\kappa)+\varepsilon}.
        \end{align}
    To establish a bound that is uniform in $\kappa$, we check when \eqref{eqn:A_regime} and \eqref{eqn:A_bound} yield the worst bounds.  For $0\leq \kappa \leq \lfloor k/2 \rfloor $, the condition \eqref{eqn:A_regime} is the most restrictive when $\kappa=\lfloor k/2 \rfloor $, in which case $k-\kappa-1 = \lceil k/2 \rceil -1$. We verify that $\lceil k/2 \rceil -1 \geq 1$ as long as $k\geq 3$. For \eqref{eqn:A_bound}, note that $\sigma$ is a decreasing function,  so the right-hand side is the largest when $\kappa = 0$.
    Hence, uniformly for $0\leq \kappa\leq \lfloor k/2\rfloor$, in the regime $H_j\geq p^{1/(\lceil k/2\rceil  -1)}$, we have, for $\abf \in A_\kappa$,
    \begin{align}\label{eqn:expsum_1}
        \sum_{0\leq \kappa \leq \lfloor k/2 \rfloor } \sum_{\abf\in A_\kappa}\prod_{j=1}^{n}|S_j(\abf;F)|^2 
        \ll_{n,\varepsilon} p^{k+1}K^{2n}H^{2n(1-\sigma(k))+\varepsilon}.
    \end{align}
(We briefly remark that if $k\leq 4$, then we must treat quadratic polynomials in $x_j$, in which case the Weyl bound gives $p^{-1/2}H_j  + p^{1/2}\log p$ (see e.g. \cite[Theorem 8.1]{IK04}). This is only nontrivial for $H_j\gg p^{1/2}\log  p$, and so we do not obtain nontrivial bounds on $S_j(\abf;F)$ in the desired range of $H_j$.)

Next fix $\lfloor k/2 \rfloor < \kappa \leq k$, and let $\abf \in A'_\kappa$, so that $a_{k},a_{k-1},...,a_{\kappa+1}\equiv 0$ and $a_\kappa \not\equiv 0 \modd p$. 
Then
    \begin{align*}
         S_j(\abf;F)= \sum_{x_j\in (N_j-H_j,N_j+H_j]} 
         \sum_{y_j\in (0,K]}   e_p(\sum_{i=0}^{\kappa} a_i  \sum_{r=i}^{k} \binom{r}{i} c_{j,r}x_j^{r-i} y_j^i).
    \end{align*}
   For fixed $x_j$, we apply Theorem \ref{thm:weyl} to the sum over $y_j$. Here, the argument is a single-variable polynomial in $y_j$ of degree $\kappa\geq  \lfloor k/2 \rfloor +1 \geq 3$ as long as $k\geq 4$. The leading coefficient is $\phi_{j,\kappa}(x_j)= a_\kappa  \sum_{r=\kappa}^{k} \binom{r}{\kappa} c_{j,r}x_j^{r-\kappa} $. This is a degree $k-\kappa$ polynomial in $x_j$, with nonzero leading coefficient $a_\kappa \binom{k}{\kappa}c_{j,k}$, and so $\phi_{j,\kappa}(x_j)=0$ in $\bF_p$ for at most $k-\kappa \leq k$ many $x_j$'s. Let $Z_{j,\kappa}\subseteq \bF_p$ denote the set of zeros of $\phi_{j,\kappa}$. Then 
        \begin{align*}
          |  S_j(\abf;F) |\leq 
          kK+
          \sum_{x_j\in (N_j-H_j,N_j+H_j]\cap (\bF_p\setminus Z_{j,\kappa})} 
        | \sum_{y_j\in (0,K]}   e_p(\sum_{i=0}^{\kappa} a_i  \sum_{r=i}^{k} \binom{r}{i} c_{j,r}x_j^{r-i} y_j^i)|.
        \end{align*}
Applying Theorem \ref{thm:weyl}, we get
    \begin{align*}
         S_j(\abf;F) \ll_{\varepsilon} kK+  H_jK^{1+\varepsilon}(p^{-1}+K^{-1}+pK^{-\kappa})^{\sigma(\kappa)}.
    \end{align*}
    In the second term, $K^{-1}\gg p^{-1}$ since $K\ll p$, while $K^{-1}\gg pK^{-\kappa}$ as long as
        \begin{align}\label{eqn:A'_regime}
            K\geq p^{1/(\kappa-1)}.
        \end{align}
  In this regime, we have the bound
    \begin{align}\label{eqn:A'_bound}
        S_j(\abf;F) \ll_{\varepsilon} k K+  H_jK^{1-\sigma(\kappa)+\varepsilon}.
    \end{align}
    Again, \eqref{eqn:A'_regime} is the most restrictive when $\kappa = \lfloor k/2 \rfloor +1$, while the bound in \eqref{eqn:A'_bound} is the worst when $\kappa = k$. Hence uniformly for $\lfloor k/2 \rfloor\leq \kappa \leq k-1$, in the regime $K\geq p^{1/\lfloor k/2 \rfloor}$, we have, for $\abf \in A'_{\kappa}$, the bound $S_j(\abf;F) \ll_{\varepsilon} k K+  H_jK^{1-\sigma(k)+\varepsilon}$.
    The second term here dominates since we assume $K\ll H_j$. So
    \begin{align} \label{eqn:expsum_2}
        \sum_{\lfloor k/2 \rfloor < \kappa \leq k} \sum_{\abf\in A_\kappa'}\prod_{j=1}^{n}|S_j(\abf;F)|^2 
      \ll_{n,k,\varepsilon} p^{k+1}H^{2n}K^{2n(1-\sigma(k))+\varepsilon}.
    \end{align}
    Combining \eqref{eqn:expsum_1}, \eqref{eqn:expsum_2} and the conditions $H_j\geq p^{1/(\lceil k/2 \rceil -1)}$ and $K\geq p^{1/\lfloor k/2 \rfloor}$ and noting that $k/2 -1 \leq \lceil k/2 \rceil -1 \leq \lfloor k/2 \rfloor$ proves the theorem.

\section*{Acknowledgements}
 The author thanks Lillian B. Pierce for her continued encouragement and support and for many helpful discussions. 
The author also thanks Damaris Schindler for her encouragement, her insights relating to the work of Birch and Schmidt, and for drawing our attention to the work of Brandes on linear spaces. We are further grateful to Rainer Dietmann for suggesting the paper of Schmidt on exponential sums. Finally, we thank Dante Bonolis, Julia Brandes, Tim Browning, Kevin Hughes, Emmanuel Kowalski, Akshat Mudgal, and Katherine Woo for helpful conversations. The author was partially supported by NSF DMS-2200470 and the Katherine Goodman Stern Fellowship from The Graduate School at Duke University for portions of this project.

\bibliographystyle{alpha}
\bibliography{_bibliography}
\end{document}